\documentclass[10.5pt,oneside,english]{amsart}
\RequirePackage{amsthm,amsmath,mathtools}

\usepackage[T1]{fontenc}
\usepackage{geometry}
\usepackage{subcaption}
\usepackage{babel,verbatim}
\usepackage{float}
\usepackage{amstext}
\usepackage{amssymb}
\usepackage{graphicx}
\usepackage[numbers]{natbib}
\RequirePackage[colorlinks,citecolor=blue,urlcolor=blue]{hyperref}

\makeatletter
\numberwithin{equation}{section}
\numberwithin{figure}{section}
\theoremstyle{plain}

\theoremstyle{plain}
\newtheorem{thm}{\protect\theoremname}
\theoremstyle{remark}
\newtheorem*{rem*}{\protect\remarkname}
\theoremstyle{remark}
\newtheorem{rem}[]{\protect\remarkname}
\theoremstyle{plain}
\newtheorem*{assumption*}{\protect\assumptionname}
\theoremstyle{plain}
\newtheorem{lem}[thm]{\protect\lemmaname}
\theoremstyle{plain}

\theoremstyle{plain}
\newtheorem{prop}[thm]{\protect\propositionname}
\theoremstyle{definition}
\newtheorem*{example*}{\protect\examplename}

\usepackage{algorithm,algpseudocode}
\usepackage{pgfplots} 
\usepackage{subcaption}

\usepackage{caption}
\makeatother
\providecommand{\algorithmname}{Algorithm}
\providecommand{\assumptionname}{Assumption}
\providecommand{\examplename}{Example}
\providecommand{\lemmaname}{Lemma}
\providecommand{\propositionname}{Proposition}
\providecommand{\remarkname}{Remark}
\providecommand{\corollaryname}{Corollary}
\providecommand{\theoremname}{Theorem}

\begin{document}
	
\title{Exact Simulation of the Extrema of Stable Processes}

\author{Jorge Gonz\'{a}lez C\'{a}zares, Aleksandar Mijatovi\'{c}, \and
	Ger\'{o}nimo Uribe Bravo}

\address{Department of Statistics, University of Warwick, \& The Alan Turing Institute, UK}

\email{jorge.gonzalez-cazares@warwick.ac.uk}

\address{Department of Statistics, University of Warwick, \& The Alan Turing Institute, UK}

\email{a.mijatovic@warwick.ac.uk}

\address{Universidad Nacional Aut\'{o}noma de M\'{e}xico, M\'{e}xico}

\email{geronimo@matem.unam.mx}
	
\begin{abstract}
We exhibit an exact simulation algorithm for the supremum of a stable
process over a finite time interval using 
dominated coupling from the past (DCFTP).
We establish a novel perpetuity equation 
for the supremum (via the representation of the concave 
majorants of L\'{e}vy processes~\citep{MR2978134})
and apply it to construct a Markov chain in the DCFTP algorithm. 
We prove that  the number of steps taken backwards in time before 
the coalescence is detected is finite. We analyse numerically the performance of 
the algorithm (the code, written in Julia 1.0, is available on GitHub).
\end{abstract}
\keywords{random variate generation; perpetuities; simulation; perfect simulation; dominated coupling from the past; stable process}

\maketitle

\section{Introduction}

This paper describes an algorithm for generating \emph{exact} samples
of the extrema of a stable process (see Algorithm~\ref{alg:Algorithm 1}
below) based on dominated coupling from the past (DCFTP), a 
coupling method for exact simulation from an invariant distribution
of a Markov chain on an ordered state space (cf.~\citep{MR1788098} and the references therein). 
The chain
in Algorithm~\ref{alg:Algorithm 1} is based on a novel characterisation
for the law of the supremum of a stable process at a fixed time in
Theorem~\ref{thm:SecondPereq}. Perpetuity~\eqref{eq:PerpEq2}
is established via the stochastic representation for concave majorants
of L\'{e}vy processes~\citep{MR2978134} and the scaling property
of stable laws (see Section~\ref{sec:Proof_Thm1} below for the proof
of Theorem~\ref{thm:SecondPereq}).

\begin{thm}
\label{thm:SecondPereq}Let $Y=(Y_t)_{t\in[0,\infty)}$ be a stable process with the stability
and positivity parameters $\alpha$ and $\rho$, respectively (see
Appendix~\ref{sec:StableSampling}). Define $\overline{Y}_{1}=\sup_{s\in[0,1]}Y_{s}$
and let $\left(B,U,V,S,\overline{Y}_{1}\right)$ be a random vector
with independent components, where $U,V$ are uniform on $(0,1)$,
$B$ is Bernoulli with parameter $1-\rho$ and $S$ has the law of
$Y_{1}$ conditioned on being positive. Then the following equality
in law holds:
\begin{eqnarray}
\overline{Y}_{1} & \overset{d}{=} &
\Lambda^{\frac{1}{\alpha}}\left(U^{\frac{1}{\alpha}}\overline{Y}_{1}+
\left(1-U\right)^{\frac{1}{\alpha}}S\right),
\label{eq:PerpEq2}
\end{eqnarray}
where $\Lambda=1+B(V^{\frac{1}{\rho}}-1)$.
Furthermore, the law of $\overline{Y}_{1}$ is the unique solution
to~\eqref{eq:PerpEq2}.
\end{thm}

The universality of stable processes makes them ubiquitous in probability
theory and many areas of statistics and natural and social sciences (see
the monograph~\citep{MR1745764} and the references therein).
The problem of efficient simulation of stable random variables in 
the context of statistics was addressed in~\citep{MR3233961}. 
Among the path properties, the running
supremum $\overline{Y}_{t}=\sup_{s\in[0,t]}Y_{s}\overset{d}{=}t^{1/\alpha}\overline{Y}_{1}$
of a stable process is of special interest (cf.~\citep{MR2453779,MR2440923,MR2789582,MR2402160})
as it arises in application areas such as optimal stopping, the prediction
of the ultimate supremum and risk theory (cf.~\citep{MR2932671,MR2453779}).

In general, one has no access to the density, distribution or even characteristic 
function of $\overline{Y}_1$,
making a rejection sampling algorithm (see~\citep[Sec.~II.3]{devroye:1986})
for $\overline{Y}_1$  difficult to construct. 
More precisely, if $Y$ has no positive jumps, the strong Markov property and the
fact that $Y$ does not jump over positive levels imply that $\overline{Y}_{1}$
has the same law as $Y_{1}$ conditioned on being positive~\citep{MR3033593}.
In all other cases, the law of $\overline{Y}_{1}$ is not accessible
in closed form and the information about it in the literature is obtained
via analytical methods based on the Wiener-Hopf factorisation. If
$Y$ has no negative jumps,~\citep{MR2440923} gives an alternating
series expression for the density, while~\citep{MR2789582,MR2402160}
give a double series representation for a dense class of parameters.
The coefficients in these representations are complicated and it is
not immediately clear how one could use them to design a simulation
algorithm. Moreover, in the general case, when $\alpha$ is rational the series representation
is proved to be convergent for finitely
many $\rho$ only~\citep{MR3005012}. 
Our simulation algorithm is based on purely probabilistic  
methods (it may be regarded as a generalization of 
the exact simulation algorithm for Vervaat perpetuities in~\citep{MR2587562})
and as such covers the entire class of stable processes. 

\subsection{Exact Simulation Algorithm\label{subsec:IntroTheAlgorithm}}
The perpetuity in~\eqref{eq:PerpEq2}
above gives rise to an update function 
$x \mapsto \phi(x,\Theta)$ 
of a Markov chain on $(0,\infty)$, 
where the components of the random  vector $\Theta$ are the random variables in Theorem~\ref{thm:SecondPereq} 
(see~\eqref{eq:phi_update} below for the precise definition of $\phi$).
The invariant distribution (i.e. invariant probability measure as defined in~\citep[p.~229]{MR2509253}) for  the chain 
$X'=\left\{ X_{n}'\right\} _{n\in\mathbb{Z}}$,
defined by 
$X_n'=\phi(X_{n-1}',\Theta_{n-1})$
with
$\left\{ \Theta_{n}\right\} _{n\in\mathbb{Z}}$
a sequence of independent copies of $\Theta$,
equals that of $\overline{Y}_1$. 
However, since 
$x \mapsto \phi(x,\Theta)$ 
is strictly increasing in $x$ with probability one, 
no coalescence occurs, making $X'$ unusable for DCFTP purposes. 
Fortunately, the structure of the perpetuity 
in~\eqref{eq:PerpEq2} is such that the update function $\phi$ can be modified to a 
\emph{multigamma coupler}~\citep{MR1650023} $x\mapsto \psi(x,\Theta)$,
which is constant on a subinterval in $(0,\infty)$ with positive probability and
globally non-decreasing. 
The definition of $\psi$, given in Lemma~\ref{lem:perpCoalescence} below,
was inspired by~\citep{MR2587562} where such a modification was applied to Vervaat perpetuities. 
The construction requires an addition of a single independent uniform random variable to the vector $\Theta$ 
and yields a Markov chain 
$X=\left\{ X_{n}\right\} _{n\in\mathbb{Z}}$
on $(0,\infty)$ 
via 
$X_{n}=\psi\left(X_{n-1},\Theta_{n-1}\right)$,
where
$\left\{ \Theta_{n}\right\} _{n\in\mathbb{Z}}$
are
independent copies of $\Theta$.
The invariant distribution of $X$ equals that of $\overline{Y}_1$
\textit{and} the coalescence occurs at every step with positive probability. 
The former follows from Theorem~\ref{thm:SecondPereq} and the fact that the chains 
$X$ and $X'$ have the same transition probabilities 
(see Lemma~\ref{lem:perpCoalescence} below)
and the latter is a consequence of the
structure of $\psi$. 

Our aim is to sample $X_{0}$, whose law equals that of $\overline{Y}_{1}$.
By construction of $\psi$ 
it follows that 
$\psi\left(x,\Theta\right)=\psi\left(a\left(\Theta\right),\Theta\right)$
for any $x\in(0, a\left(\Theta\right)]$, 
where $\theta\mapsto a\left(\theta\right)$ is a positive deterministic function
explicitly given in~\eqref{eq:varphi_and_a} of 
Lemma~\ref{lem:perpCoalescence} below.
The coalescence for $X$ occurs every time the inequality $X_n\leq a\left(\Theta_n\right)$ is satisfied,
since,
if $-\sigma$ is such a time, then $X_{-\sigma+1}=\psi\left(a\left(\Theta_{-\sigma}\right),\Theta_{-\sigma}\right)$
disregards the value
$X_{-\sigma}$
and hence the entire trajectory of $X$ prior to time $-\sigma+1$. 

The task now is to detect whether the event 
$\left\{X_n\leq a\left(\Theta_n\right)\right\}$ occurred without knowing the value of $X_n$
(if we had access to $X_n$ for any $n\in\mathbb{Z}$, we would have a sample from the law of $\overline{Y}_1$!).
DCFTP~\citep{MR1788098} suggests to look for 
a process $D=\left\{ D_{n}\right\}_{n\in\mathbb{Z}}$ satisfying $D_{n}\geq X_{n}$
for all $n\in\mathbb{Z}$, which can be simulated backwards in time (starting at $0$) together with the 
i.i.d. sequence  $\left\{ \Theta_{n}\right\} _{n\in\mathbb{Z}}$. 
It is possible to define such a process $D$, which turns out to be stationary but non-Markovian, by ``unwinding''
the recursion for $X$ backwards in time and bounding the terms 
(see~\eqref{eq:DominatingProcess} in Sec.~\ref{sec:SimulationDifficultProcesses}). 

\begin{algorithm}
\caption{Exact sampling from the law of $\overline{Y}_1$}
\label{alg:Algorithm 1} 
\begin{algorithmic}[1]
	\State{Starting at $0$, sample $\{(D_n,\Theta_n)\}_{n\in\mathbb{Z}}$ backwards in time until 
	$-\sigma=\sup\{n\leq 0:D_n\leq a(\Theta_n)\}$}
	\State{Put $X_{-\sigma+1}=\psi(a(\Theta_{-\sigma}),\Theta_{-\sigma})$}
	\State{Compute recursively $X_n=\psi(X_{n-1},\Theta_{n-1})$ for $n=-\sigma+2,\ldots,0$}
	\State{\Return $X_0$}
\end{algorithmic}
\end{algorithm}

The backward simulation of $\left\{ \left(D_{n},\Theta_{n}\right)\right\}_{n\in\mathbb{Z}}$ 
in step $1$ of Algorithm~\ref{alg:Algorithm 1} is discussed in Section~\ref{sec:Backward_Sim} below.
It relies on two ingredients: (A) the simulation of the indicators of independent events with summable probabilities 
and (B) the simulation of a random walk with negative drift \textit{and} its future supremum. 
By the Borel-Cantelli lemma, only finitely many indicators in (A) are non-zero. A simple and efficient algorithm for the
simulation of the entire sequence is given in Section~\ref{subsec:StableInc} below. The algorithm for (B)
has been developed in~\citep[Sec.~4]{MR2865624}. For completeness, in Section~\ref{subsec:BlanchetSigman} below 
we present the algorithm from~\citep[Sec.~4]{MR2865624} applied to the specific random walk that arises in 
definition~\eqref{eq:DominatingProcess} of our dominating process $D$.
The algorithm in~\citep[Sec.~4]{MR2865624}  
requires the simulation of the walk under the original measure as well as under an exponential change of measure. 
In our case the increments of the random walk in question are shifted negative exponential random variables. 
This makes the dynamics of the walk explicit and easy to simulate under both measures (see Section~\ref{subsec:BlanchetSigman} below for details), 
making the implementation of Algorithm~\ref{alg:Algorithm 1} quite
fast. More precisely, Algorithm~\ref{alg:FullAlgorithm} below (a version of Algorithm~\ref{alg:Algorithm 1}) 
was implemented in Julia, 
see the GitHub repository~\cite{Jorge_GitHub} for the code and a simple user guide. 
This implementation outputs approximately $10^4$ samples every $1.15$ seconds 
(see Section~\ref{sec:numerics} for details).

Note that the random time $\sigma$ in Algorithm~\ref{alg:Algorithm 1} dictates the number 
of simulations, as steps 2-4 in the algorithm require only deterministic
computation. In order to prove that $\sigma$ is finite, 
we couple $D$ with a dominating process $D'$, which is a component of a multi-dimensional
positive Harris recurrent Markov chain $\Xi$
(see~\eqref{eq:DoubleDominatingProcess} for the definition of $D'$ and Lemma~\ref{lem:recurrenceDomination}
of Section~\ref{sec:SimulationDifficultProcesses} below). Note that we need not  
be (and in fact are not) able to simulate
$D'$. 
We apply the general state space Markov chain theory~\citep{MR2509253,MR758799} to prove 
the following result 
(see Section~\ref{proof:FiniteExpectedTerminationTime} below for details).

\begin{thm}
\label{thm:FiniteExpectedTerminationTime} The random time $\sigma$ in Algorithm~\ref{alg:Algorithm 1}
is finite a.s. Moreover, $\mathbb{E} [\sigma|\Xi_0]<\infty$ a.s. 
\end{thm}

In~\citep[Thm~5.1]{MR2587562} the authors provide a sharp estimate on $\mathbb{E} [\sigma]$ for an analogous algorithm 
in the context of Vervaat perpetuities.
Their analysis is based on the fact that their dominating process $D$ is a birth-death Markov chain and is hence
time-reversible  with skip-free increments and an explicit
invariant distribution (shifted geometric). 
In the context of Theorem~\ref{thm:FiniteExpectedTerminationTime}, 
the dominating process $D$ is non-Markovian, 
its increments are diffuse, have heavy tails and the multi-dimensional Markov chain 
$\Xi$ used to bound $D$ has a non-explicit invariant probability measure $\pi$ 
(which also has heavy tails). These heavy tails make the chain frequently take 
large values, which in turn makes the coalescence events and probabilities harder 
to trace, bound and control. Moreover, the law of the time-reversal of $\Xi$ 
(with respect to $\pi$) is very different from that of $\Xi$. The key step 
in the proof of Theorem~\ref{thm:FiniteExpectedTerminationTime} is provided 
by~\citep[Thm~8.1.1]{MR758799}, which allows us to conclude that the time-reversed chain has a Harris recurrent modification. However, 
a quantitative bound on the expected number of steps taken backwards in 
time in Algorithm~\ref{alg:Algorithm 1} remains an open problem. 

\subsection{Related literature}
Exact simulation algorithms for various instances of a general perpetuity equation $\mathcal{X}\overset{d}{=}A_0 \mathcal{X}+A_1$ (with $(A_0,A_1)$ and 
$\mathcal{X}$ independent) have been developed in the literature.

Paper~\citep{MR2587562} studies the case $A_0=A_1\geq0$, $\mathbb{E} [A_0]<1$,
specialising to the Vervaat perpetuity for $A_0=U^{1/\beta}$ with $U$ uniform on 
$(0,1)$ and $\beta\in(0,\infty)$, see also~\cite{MR3683256,MR1833738}. 
Briefly put,~\citep{MR2587562} first identifies the update function and 
constructs a multigamma coupler. The identified dominating process is 
a simple random walk with a partially absorbing barrier and whose invariant law 
is that of a shifted geometric random variable. A sped up version of a 
DCFTP algorithm~\cite{MR1833738}  in the case $\beta=1$  (i.e. when 
$\mathcal{X}$ follows the Dickman distribution) is given in~\citep{MR2575452}.

In~\citep{Devroye:2011:DCM:1899396.1899398}, the authors develop the double 
CFTP algorithm in the case $A_0=V$ and $A_1=(1-V)Z$, where $V$ takes values 
in $[0,1]$ (and has a computable density) and $Z$ is independent of $V$ with 
support in an interval $[0,c]$ for some $c<\infty$.
This structure appears similar to perpetuity~\eqref{eq:PerpEq1} of 
Propostion~\ref{prop:FirstPereq} below, where $A_0=U^{1/\alpha}$ and 
$A_1=(1-U)^{1/\alpha}\max\{Y_1,0\}$ with $Y_1$ an $\alpha$-stable random 
variable independent of the uniform $U$. 
Proposition~\ref{prop:FirstPereq} provides a key step in the proof of 
Theorem~\ref{thm:SecondPereq} above, which in turn is the cornerstone of 
Algorithm~\ref{alg:Algorithm 1}. The upper bound $c$ on the support of 
$Z$ in~\citep{Devroye:2011:DCM:1899396.1899398} is inversely proportional 
to the coalescence probability of the chain in the double CFTP algorithm, 
making its direct application to perpetuity~\eqref{eq:PerpEq1} impossible,
since $\max\{Y_1,0\}$ not only has infinite support but also a heavy tail.
Moreover, even if we could construct a stochastic (rather than constant) 
upper bound on the relevant support, this bound would necessarily still 
have a heavy tail making the coalescence in a generalisation of the 
algorithm in~\citep{Devroye:2011:DCM:1899396.1899398} unlikely. This would 
then yield long (possibly infinite) running times for such a generalisation.

Paper~\citep{ExactGralVervaatPerp} studies the generalised Vervaat perpetuity where 
$A_1=A_0A_2$ for independent $A_2$ and $A_0=U^{1/\beta}$ with $U$ uniform on 
$(0,1)$. By calculating the Laplace transform from the perpetuity, it is shown 
in~\citep{ExactGralVervaatPerp} that 
$\mathcal{X}$ has the law of the marginal of a pure jump L\'{e}vy process 
at time $\beta$ with L\'{e}vy density 
$\nu(dx)= |x|^{-1}(\mathbb{P}(A_2>x)1_{x>0}+\mathbb{P}(A_2<-x)1_{x<0})dx$. 
Techniques similar to those in~\citep{MR3024612}, based on infinite divisibility, 
are used to devise the simulation algorithm under the conditions $A_2\geq0$ and 
$\lim_{x\downarrow0}\mathbb{P}(A_2\leq x)/x<\infty$, 
without relying on Markov chain techniques. The calculation of Laplace
transforms based on perpetuities~\eqref{eq:PerpEq1} or~\eqref{eq:PerpEq2},
yields complicated equations for the Laplace transform. Furthermore,
even if we could solve for the Laplace transform of $\overline{Y}_1$,
we could not follow the simulation approach from \citep{ExactGralVervaatPerp}
as $\overline{Y}_1$ is typically not infinitely divisible.

In~\citep{MR2865624} the authors use a version of a 
multigamma coupler, allowing 
$A_1$ to have a heavy tail but assuming the independence of $A_0$ and $A_1$, 
a requirement clearly violated by perpetuities~\eqref{eq:PerpEq2} 
and~\eqref{eq:PerpEq1} in the present paper. Moreover, a certain 
domination condition~\citep[Eq.~(2) in Assumption~(B)]{MR2865624} 
for the density of $A_1$ is stipulated, which plays an important 
role in constructing the coalescence probability. This dominating condition is 
hard to establish for the density of a stable law conditioned on 
being positive, appearing in perpetuity~\eqref{eq:PerpEq2}. 
Thus, even if one could remove the assumption on the independence 
of $A_0$ and $A_1$ in~\citep{MR2865624}, this technical requirement would make it hard 
to apply directly the sampling algorithm from~\citep{MR2865624} in our setting.

The structure of the multigamma coupler used in the present paper
is closer to the one in~\citep{MR2587562} (see 
also Section~\ref{subsec:IntroTheAlgorithm} above and 
Lemma~\ref{lem:perpCoalescence} below) than the one in~\citep{MR2865624}. 
Despite the differences between the samplers in~\citep{MR2865624} and the one used here, the
construction of our dominating process was inspired by the one presented
in~\citep{MR2865624}. However, we were unable to use directly the dominating process
$V^+_k$ in~\citep[Eq.~(9)]{MR2865624}, 
which appears to be bounded from below by the deterministic function 
$k\mapsto e^{ak/2}/(1-e^{-a/2})$ (for all positive integers $k$ and 
some constant $a>0$) tending to infinity exponentially fast 
and hence suggesting a positive probability of never detecting coalescence.
It appears that this issue could be circumvented in the 
general context of~\citep{MR2865624} by a simple adaptation of our 
dominating process defined in~\eqref{eq:DominatingProcess} below, 
which is 
based on the idea of adaptive bounds 
(cf. Figure~\ref{fig:adaptive_bound}).

A perpetuity can be understood as the special case of the stochastic
fixed point equation $\mathcal{X}\overset{d}{=}f(\mathcal{X},U)$ 
in a general state space for independent $\mathcal{X}$ and $U$ and some 
measurable function $f$. See the monograph~\citep{MR3443710} for a 
comprehensive survey on the variety of Markov chain techniques, 
such as CFTP and DCFTP, used to obtain exact samples of $\mathcal{X}$.

The problem of the exact simulation of the first passage event of a 
spectrally positive stable process (resp. a L\'evy process with infinite 
activity and finite variation) is addressed in~\citep{Chi1} 
(resp.~\citep{Chi2}). Algorithm~\ref{alg:Algorithm 1} solves this 
problem for all stable processes as follows: for any $x>0$,
define the \emph{first passage time} $\tau_x:=\inf\{t>0:Y_t\geq x\}$ 
and note that the equality of events $\{\tau_x>t\}=\{\overline{Y}_t<x\}$
for all $t\in(0,\infty)$ and the scaling property yield the equality in law 
$\tau_x\overset{d}{=}(x/\overline{Y}_1)^ \alpha$. 

We conclude the introduction by noting that Proposition~\ref{prop:FirstPereq} 
easily implies the asymptotic behaviour at infinity of the distribution function 
of $\overline{Y}_1$ stated in~\citep[Prop.~VIII.1.4, p.~221]{MR1406564}.
Excluding the spectrally negative case, perpetuity~\eqref{eq:PerpEq1}
and the Grincevi\u{c}ius-Grey theorem~\citep[Thm~2.4.3]{MR3497380} yield
$\lim_{x\to\infty}2\mathbb{P}\left(Y_1U^{1/\alpha}>x\right)/
\mathbb{P}\left(\overline{Y}_1>x\right) =1$.
By Breiman's lemma~\citep[Lem.~B.5.1]{MR3497380} we have 
$\lim_{x\to\infty} 2\mathbb{P}\left(Y_1U^{1/\alpha}>x\right)/\mathbb{P}\left(Y_1>x\right)=1$,
implying $\lim_{x\to\infty} \mathbb{P}\left(\overline{Y}_1>x\right)/x^{-\alpha} =\Gamma(\alpha)\sin(\pi\alpha\rho)/\pi$
via the classical tail behaviour of the stable law~\citep[Sec.~4.3]{MR1745764}. 

The remainder of the paper is structured as follows. In Section~\ref{sec:Proof_Thm1},
we establish perpetuity~\eqref{eq:PerpEq1} and apply it in the proof of 
Theorem~\ref{thm:SecondPereq}. In Section~\ref{sec:SimulationDifficultProcesses}
we define the update function $\psi$ (in Lemma~\ref{lem:perpCoalescence}), 
construct the dominating process and prove
Theorem~\ref{thm:FiniteExpectedTerminationTime} above. Section~\ref{sec:Backward_Sim}
discusses the backward simulation of $\{(D_n,\Theta_n)\}_{n\in\mathbb{Z}}$. 
Finally, a numerical performance analysis is found in Section~\ref{sec:numerics}.

\section{\label{sec:Proof_Thm1}Stochastic Perpetuities}

Let $Y$ be a stable process with stability and positivity parameters $\alpha$ and 
$\rho$, respectively (see Appendix~\ref{sec:StableSampling} below for definition). 
Since $Y_0=0$ and the scaling property yield
$\overline{Y}_{t}=\sup_{s\in\left[0,t\right]}Y_{s}\overset{d}{=}t^{1/\alpha}\overline{Y}_{1}$
for all $t\in[0,\infty)$,
we may restrict our attention to $\overline{Y}_1$.
Let $S\left(\alpha,\rho\right)$ and 
$\overline{S}\left(\alpha,\rho\right)$
denote the laws of 
$Y_{1}$ and  $\overline{Y}_{1}$, respectively.  
Since $\mathbb{P}\left(Y_{t}>0\right)=\rho$ for any $t>0$,
the extreme cases $\rho\in\left\{ 0,1\right\}$ are excluded from our analysis 
as they correspond to $Y$ having monotone paths.
Let $U(0,1)$ denote the uniform law on $(0,1)$
and define
$x^+=\max\{x,0\}$ for any real number $x\in\mathbb{R}$.

\begin{prop}
\label{prop:FirstPereq}Let $\left(\overline{Y}_1,Z,U\right)\sim \overline{S}\left(\alpha,\rho\right)\times S\left(\alpha,\rho\right)\times U\left(0,1\right)$.
Then the law of $\overline{Y}_{1}$ is the unique solution of the following perpetuity: 
\begin{eqnarray}
\overline{Y}_{1} & \overset{d}{=} & U^{\frac{1}{\alpha}}\overline{Y}_{1}+\left(1-U\right)^{\frac{1}{\alpha}}Z^{+}.\label{eq:PerpEq1}
\end{eqnarray}
\end{prop}

To prove this result, we need the next definition. 
For any  $a<b$, the \emph{concave majorant}
of a function $f:[a,b]\to\mathbb{R}$
is defined as the smallest concave function $c:[a,b]\to\mathbb{R}$, 
such that $c\left(t\right)\geq x\left(t\right)$ for every $t\in\left[a,b\right]$.
The proof of Proposition~\ref{prop:FirstPereq} exploits the fact that 
the supremum of a function lies on its concave majorant, 
at the end of all (if any) faces with positive slope.
Following the classical result for the complete description
of a concave majorant of random walks, \citep{MR2978134} describes 
the continuous time analogue of these results for L\'evy processes
(\citep{MR2978134} is phrased in terms of the convex minorant, 
but through a change of sign their results cover the concave majorant). 
The idea is as follows: fix a sample path of $Y$ and pick a random face of its concave majorant 
above an independent uniform point in $[0,1]$. 
The length of the chosen face is distributed as $V\sim U(0,1)$ 
and its height 
is distributed as the increment of a stable process over a time interval of duration $V$.
Moreover, after removing this face (together with the path underneath it)
the remainder of the concave majorant  behaves like a concave majorant
of a stable process over the time interval $\left[0,1-V\right]$, 
see~\citep{MR2978134}. 
This recursive relation and the scaling property of $Y$ will yield the perpetuity in~\eqref{eq:PerpEq1}.

\begin{proof}

A stick-breaking process 
$\left\{ \ell_{n}\right\} _{n\geq1}$
on $\left[0,1\right]$ 
is defined recursively as follows: 
\[
\ell_{n}=V_{n}\left(1-L_{n-1}\right),\quad n\geq1,
\]
where
$L_{n-1}=\ell_{1}+\cdots+\ell_{n-1}$,
$L_0=0$
and
$\left\{ V_{n}\right\} _{n\geq1}$ is a sequence of i.i.d. random variables with law $U(0,1)$
(independent of $Y$).
Let $C=(C_t)_{t\in[0,1]}$ be the  concave majorant  of the L\'{e}vy process $Y$. 
Let $\left(d_{n}-g_{n},C_{d_{n}}-C_{g_{n}}\right)_{n\geq1}$ be the
lengths and heights of the faces of $C$ picked at random, uniformly
on lengths and without replacement ($g_n$ and $d_n$ denote the  beginning and end times for 
the $n$-th face). 
\citep[Thm.~1]{MR2978134}
asserts the equality in law
\[
\left(d_{n}-g_{n},C_{d_{n}}-C_{g_{n}}\right)_{n\geq1}\overset{d}{=}\left(\ell_{n},Y_{L_{n}}-Y_{L_{n-1}}\right)_{n\geq1}.
\]

The concave majorant $(C_t)_{t\in[0,1]}$ is piecewise linear, with the corresponding slopes forming
a non-increasing piecewise constant function in $t$. 
Hence 
$\overline{Y}_{1}$ is always contained in the image of the function $C$. Moreover,  the supremum
equals the sum of all the positive heights of $C$: 
\begin{eqnarray*}
\overline{Y}_{1} & = & \sum_{n=1}^{\infty}\left(C_{d_{n}}-C_{g_{n}}\right)^{+}\overset{d}{=}\sum_{n=1}^{\infty}\left(Y_{L_{n}}-Y_{L_{n-1}}\right)^{+}.
\end{eqnarray*}
Conditional on $\left\{ L_{n}\right\}_{n\geq1}$, the random variables
$\left\{ Y_{L_{n}}-Y_{L_{n-1}}\right\}_{n\geq1}$ are independent and have
the same distribution as the respective $Y_{\ell_{n}}$. Hence, for an
independent i.i.d. sequence $\left\{ Z_{n}\right\} _{n\geq1}$ with
law $S\left(\alpha,\rho\right)$ we have
\[
\left(\ell_{n},Y_{Z_{n}}-Y_{Z_{n-1}}\right)_{n\geq1}\overset{d}{=}\left(\ell_{n},\ell_{n}^{\frac{1}{\alpha}}Z_{n}\right)_{n\geq1},
\]
implying
\begin{eqnarray}
\label{eqn:IL}
\overline{Y}_{1} & \overset{d}{=}  &\sum_{n=1}^{\infty}\left(Y_{L_{n}}-Y_{L_{n-1}}\right)^{+}\overset{d}{=}\sum_{n=1}^{\infty}\ell_{n}^{\frac{1}{\alpha}}Z_{n}^{+}.
\end{eqnarray}

It is well-known that 
$\left\{ \frac{\ell_{n}}{1-\ell_{1}}\right\} _{n\geq2}$ is a stick-breaking
process on $\left[0,1\right]$,
independent of $\ell_1\sim U(0,1)$ (and $\left\{ Z_{n}\right\} _{n\geq1}$).  
Hence by~\eqref{eqn:IL} we find the equality in law
\[
\overline{Y}_{1}\overset{d}{=}\sum_{n=2}^{\infty}\left(\frac{\ell_{n}}{1-\ell_{1}}\right)^{\frac{1}{\alpha}}Z_{n}^{+},
\]
which, together with~\eqref{eqn:IL}, implies the perpetuity 

\begin{eqnarray*}
\overline{Y}_{1} & \overset{d}{=} & \ell_{1}^{\frac{1}{\alpha}}Z_{1}^{+}+\left(1-\ell_{1}\right)^{\frac{1}{\alpha}}\overline{Y}_{1}.
\end{eqnarray*}
Finally, the uniqueness of solution follows from~\citep[Thm~2.1.3]{MR3497380}.
\end{proof}

Let $S^{+}\left(\alpha,\rho\right)$ denote the law of $Y_{1}$ conditioned on 
being positive. For $n,m\in\mathbb{Z}$ define the sets
\begin{equation}
\label{eq:def_mathcal_Z}
\mathcal{Z}^{n} = \left\{ k\in\mathbb{Z}:k<n\right\} ,\quad\mathcal{Z}_{m}^{n}=\mathcal{Z}^{n}\backslash\mathcal{Z}^{m}.
\end{equation}

\begin{proof}[\label{proof:Theorem1}Proof of Theorem~\ref{thm:SecondPereq}]
Note that the random variable $Z^{+}$ in Propostiion~\ref{prop:FirstPereq} 
behaves like the product of a Bernoulli random variable and a stable random variable conditioned
on being positive, i.e.,  if $B\sim Ber\left(\rho\right)$ and $S\sim S^{+}\left(\alpha,\rho\right)$
are independent, then $Z^{+}\overset{d}{=}BS$. 
Since $\mathbb{P}\left(Z^{+}=0\right)=1-\rho>0$,
the idea behind the proof of 
Theorem~\ref{thm:SecondPereq}
is to iterate perpetuity~\eqref{eq:PerpEq1}
backwards in time until the first time we observe $Z^{+}>0$.

More precisely, by Proposition~\ref{prop:FirstPereq} and Kolmogorov's consistency theorem we can construct
a stationary Markov chain $\left\{ \left(U_{n},Z_{n},\zeta_{n}\right)\right\} _{n\in\mathcal{Z}^{1}}$
with invariant law $U\left(0,1\right)\times S\left(\alpha,\rho\right)\times\overline{S}\left(\alpha,\rho\right)$,
where $\left\{ \left(U_{n},Z_{n}\right)\right\} _{n\in\mathcal{Z}^{1}}$
is an i.i.d. sequence with law $U\left(0,1\right)\times S\left(\alpha,\rho\right)$ 
and 
\begin{eqnarray*}
\zeta_{n+1} & = & U_{n}^{\frac{1}{\alpha}}Z_{n}^{+}+\left(1-U_{n}\right)^{\frac{1}{\alpha}}\zeta_{n},\quad n\in\mathcal{Z}^{0}.
\end{eqnarray*}
Define $V_{0}=1$ and  $V_{n}=\prod_{m\in\mathcal{Z}_{n}^{0}}\left(1-U_{m}\right)$
for $n\in\mathcal{Z}^{0}$. Then the following equality holds 
\begin{equation}
\zeta_{0}=\sum_{m\in\mathcal{Z}_{n}^{0}}\left(U_{m}V_{m+1}\right)^{\frac{1}{\alpha}}Z_{m}^{+}+V_{n}^{\frac{1}{\alpha}}\zeta_{n}\qquad
\text{ for all $n\in\mathcal{Z}^{0}$.}
\label{eq:nth-step}
\end{equation}
Let
$\tau=\sup\left\{ n\in\mathcal{Z}^{0}:Z_{n}>0\right\}$
(with convention $\sup\emptyset=-\infty$)
be the last time we see a positive value in the sequence
$\left\{ Z_{n}\right\}_{n\in\mathcal{Z}^{0}}$.
Substituting 
$n=\tau$ in equation~\eqref{eq:nth-step}, we get
\begin{equation}
\label{eq:as_equality}
\zeta_{0}=V_{\tau+1}^{\frac{1}{\alpha}} \left(\left(1-U_{\tau}\right)^{\frac{1}{\alpha}}\zeta_{\tau}+U_{\tau}^{\frac{1}{\alpha}}Z_{\tau}\right).
\end{equation}
This equality of course yields the same equality in law. It will hence imply the perpetuity in~\eqref{eq:PerpEq2},
if we prove that the random variables involved have the desired laws and independence structure. 

The events $\left\{ Z_{n}>0\right\}$, $n\in\mathcal{Z}^0$,
are independent with probability $\rho$,  making $\tau$  a  geometric random variable on $\mathcal{Z}^0$ with parameter $\rho$.
By construction, the coordinates of the vector $(U_n,Z_n,\zeta_n)$ are independent for any $n\in\mathcal{Z}^0$.
Hence we have 
$\left(U_{\tau},Z_{\tau},\zeta_{\tau}\right)\sim U\left(0,1\right)\times S^{+}\left(\alpha,\rho\right)\times\overline{S}\left(\alpha,\rho\right)$.
Moreover,
$\left(U_{\tau},Z_{\tau},\zeta_{\tau}\right)$
is independent of $\left(\tau,V_{\tau+1}\right)$.
Hence~\eqref{eq:as_equality} will imply the perpetuity in the theorem if we prove that 
$\Lambda$ has the same law as 
$V_{\tau+1}$.
Put differently, 
as $\tau$ and $U_0$ are independent, 
it is sufficient to prove the following equality in law
\begin{equation}
\label{eq:Distrib_identity}
V_{\tau+1}\overset{d}{=}1_{\tau=-1}+1_{\tau\neq-1}U_{0}^{\frac{1}{\rho}}=1+1_{\tau\neq-1}\left(U_{0}^{\frac{1}{\rho}}-1\right).
\end{equation}

Since $-\log\left(1-U_{1}\right)\sim Exp\left(1\right)$
is exponential with mean one,
$-\log\left(V_{n}\right)$ is gamma distributed with density
$x\mapsto x^{-n-1}e^{-x}/(-n-1)!$  for any $n\in\mathcal{Z}^0$. 
Hence, on the event $\{\tau\neq -1\}$,
the density of the conditional law 
$\left.-\log\left(V_{\tau+1}\right)\right|\tau$
is given by
$x\mapsto x^{-\tau-2}e^{-x}/(-\tau-2)!$. 
Thus, the conditional law 
$\left.-\log\left(V_{\tau+1}\right)\right|\{\tau\neq -1\}$
is exponential with density 
\begin{equation}
\label{eq:log_Lambda_exp}
x\mapsto \frac{1}{1-\rho}\sum_{k=2}^{\infty}\rho\left(1-\rho\right)^{k-1}\frac{x^{k-2}}{\left(k-2\right)!}e^{-x}=\rho e^{-\rho x},\qquad x>0.
\end{equation}
Since $-\log\left(V_{\tau+1}\right)$ takes the value $0$ when $\tau=-1$,
which happens with probability $\rho$, and is otherwise exponential with mean $1/\rho$, 
the distributional identity in~\eqref{eq:Distrib_identity} follows.

Finally, the uniqueness of the solution for perpetuity~\eqref{eq:PerpEq2} 
follows from~\citep[Thm~2.1.3]{MR3497380}.
\end{proof}

\section{\label{sec:SimulationDifficultProcesses} The Markov chain $X$ and the dominating process $D$ in Algorithm~\ref{alg:Algorithm 1}}

Let
$\mathcal{A}=\left(0,\infty\right)\times\left(0,1\right)\times\left(0,1\right) \times\left(0,1\right]$
and define the function $\phi:\left(0,\infty\right)\times\mathcal{A}\to\left(0,\infty\right)$
by
\begin{eqnarray}
\label{eq:phi_update}
\phi\left(x,\theta\right) & = &
\lambda^{\frac{1}{\alpha}}\left(u^{\frac{1}{\alpha}}x+\left(1-u\right)^{\frac{1}{\alpha}}s\right),\qquad
x\in\left(0,\infty\right),\quad\theta=(s,u,w,\lambda)\in\mathcal{A}.  
\end{eqnarray}
Note that the map $x\mapsto \phi\left(x,\theta\right)$ is increasing and linear in $x$ for all $\theta\in\mathcal{A}$ and does not depend on $w$.
Let 
$W\sim U\left(0,1\right)$ be independent of random variables 
$S$, $U$, and $\Lambda$ defined in Theorem~\ref{thm:SecondPereq}. 
Then, by Theorem~\ref{thm:SecondPereq}, 
we have
$\zeta\overset{d}{=}\phi\left(\zeta,\Theta\right)$, 
where
$\zeta\sim\overline{S}\left(\alpha,\rho\right)$ 
is independent of
$\Theta=\left(S,U,W,\Lambda\right)$.
Hence a Markov chain with the update function $\phi$ has the correct invariant law
but does not allow for coalescence: if for any $x,y\in(0,\infty)$ we have $\phi(x,\Theta)=\phi(y,\Theta)$,
by~\eqref{eq:phi_update} it follows $x=y$.
But the structure of $\phi$ and the additional randomness in $W$ allow us to modify 
the update function $x \mapsto \phi(x,\theta)$ so that coalescence can be achieved,
while keeping the law of the chain unchanged. 

\begin{lem}
\label{lem:perpCoalescence}Define the functions $\psi:\left(0,\infty\right)\times\mathcal{A}\to\left(0,\infty\right)$
and $a:\mathcal{A}\to\left(0,\infty\right)$ by the formulae
\begin{eqnarray}
\psi\left(x,\theta\right) & = & 1_{\left\{ a\left(\theta\right)\geq x\right\}
}w^{\frac{1}{\alpha\rho}}\left(1-u\right)^{\frac{1}{\alpha}}s+1_{\left\{
a\left(\theta\right)<x\right\}
}\lambda^{\frac{1}{\alpha}}\left(u^{\frac{1}{\alpha}}x+\left(1-u\right)^{\frac{1}{\alpha}}s\right),
\label{eq:varphi_and_A}\\
a\left(\theta\right) & = &
\left(\lambda^{-\frac{1}{\alpha}}-1\right)\left(\frac{1-u}{u}\right)^{\frac{1}{\alpha}}s.
\label{eq:varphi_and_a}
\end{eqnarray}
The map $x\mapsto\psi(x,\theta)$ is non-decreasing in $x$ for all $\theta\in\mathcal{A}$. 
Moreover, for $\zeta$ and $\Theta$
as in the paragraph above, we have
$\phi\left(x,\Theta\right)\overset{d}{=}\psi\left(x,\Theta\right)$
for all $x>0$ and  
$\overline{S}\left(\alpha,\rho\right)$
is the unique solution of the distributional equation
$\zeta\overset{d}{=}\psi\left(\zeta,\Theta\right)$.
\end{lem}

\begin{proof}
The function $\psi$ 
takes constant value of $w^{\frac{1}{\alpha\rho}}\left(1-u\right)^{\frac{1}{\alpha}}s$
for $x\in\left(0,a\left(\theta\right)\right]$ and increases linearly
on the interval $\left(a\left(\theta\right),\infty\right)$ with the right limit 
satisfying 
$\lim_{x\searrow a\left(\theta\right)}\psi\left(x,\theta\right)=\left(1-u\right)^{\frac{1}{\alpha}}s>\psi\left(a\left(\theta\right),\theta\right)$.
Hence the desired monotonicity follows.

We now prove that $\phi\left(x,\Theta\right)\overset{d}{=}\psi\left(x,\Theta\right)$
for all $x>0$, i.e the transition probabilities for the update functions $\phi$ and $\psi$ coincide.
Pick $x>0$ and note that $\left\{ \phi\left(x,\Theta\right)=\psi\left(x,\Theta\right)\right\} 
\supset\left\{ a\left(\Theta\right)<x\right\}$.
Thus, for any $y>0$ we have
$\mathbb{P}\left( \phi\left(x,\Theta\right)\leq y,
a\left(\Theta\right)<x \right)=\mathbb{P}\left(
\psi\left(x,\Theta\right)\leq y,
a\left(\Theta\right)<x \right)$.  
Define 
\[
v\left(u,s\right)=\left(\frac{\left(1-u\right)^{\frac{1}{\alpha}}s}{u^{\frac{1}{\alpha}}x+\left(1-u\right)^{\frac{1}{\alpha}}s}\right)^{\alpha\rho}\in\left(0,1\right),
\]
and note that $\left\{ a\left(\Theta\right)\geq x\right\} =\left\{ \Lambda^{\rho}\leq v\left(U,S\right)\right\}$.
On this event, the definition of $\Lambda$ in Theorem~\ref{thm:SecondPereq}
implies the inequality $\Lambda<1$, in which case $\Lambda^{\rho}$
is uniform on 
$(0,1)$. 
Hence the conditional law of $\Lambda$, given 
$\left(U,S\right)$ and $\left\{ a\left(\Theta\right)\geq x\right\} $,
is uniform on the interval 
$(0,v\left(U,S\right))$.
Moreover, the conditional law of $v\left(U,S\right)W$, given  
$\left(U,S\right)$ and on $\left\{ a\left(\Theta\right)\geq x\right\}$,
is also uniform on 
$(0,v\left(U,S\right))$.
Hence
for any $y>0$
the following equalities hold:
\begin{eqnarray*}
\mathbb{P}\left(\left. \phi\left(x,\Theta\right)\leq y, a\left(\Theta\right)\geq x \right|U,S\right)
 & = & \mathbb{P}\left(\left. \Lambda^{\rho}\leq\left(\frac{y}{U^{\frac{1}{\alpha}}x+\left(1-U\right)^{\frac{1}{\alpha}}S}\right)^{\alpha\rho}, a\left(\Theta\right)\geq x \right|U,S\right)\\
 & = & \mathbb{P}\left(\left. v\left(U,S\right)W\leq\left(\frac{y}{U^{\frac{1}{\alpha}}x+\left(1-U\right)^{\frac{1}{\alpha}}S}\right)^{\alpha\rho}, a\left(\Theta\right)\geq x \right|U,S\right)\\
 & = & \mathbb{P}\left(\left. W^{\frac{1}{\alpha\rho}}\left(1-U\right)^{\frac{1}{\alpha}}S\leq y, a\left(\Theta\right)\geq x\ \right|U,S\right)\\
 & = & \mathbb{P}\left(\left. \psi\left(x,\Theta\right)\leq y, a\left(\Theta\right)\geq x \right|U,S\right).
\end{eqnarray*}
Taking expectations in this identity yields the unconditional equality 
$\mathbb{P}\left(\phi\left(x,\Theta\right)\leq y, a\left(\Theta\right)\geq x \right)
=\mathbb{P}\left(\psi\left(x,\Theta\right)\leq y, a\left(\Theta\right)\geq x\right)$.
Hence we get
$\mathbb{P}\left(\phi\left(x,\Theta\right)\leq y\right) = \mathbb{P}\left(\psi\left(x,\Theta\right)\leq y\right)$
for all $y>0$, implying the equality in law
$\phi\left(x,\Theta\right)\overset{d}{=}\psi\left(x,\Theta\right)$
for arbitrary $x>0$.

Pick $y>0$.
Since 
$\Theta$ and $\zeta$ 
are independent, 
by Theorem~\ref{thm:SecondPereq} 
we have
\begin{eqnarray*}
\mathbb{P}\left(\zeta\leq y\right)=\mathbb{P}\left(\phi\left(\zeta,\Theta\right)\leq y\right) & = & \int_{[0,\infty)}\mathbb{P}\left(\phi\left(x,\Theta\right)\leq y\right)\mathbb{P}\left(\zeta\in dx\right)\\
 & = & \int_{[0,\infty)}\mathbb{P}\left(\psi\left(x,\Theta\right)\leq y\right)\mathbb{P}\left(\zeta\in dx\right)=\mathbb{P}\left(\psi\left(\zeta,\Theta\right)\leq y\right),
\end{eqnarray*}
implying
$\zeta\overset{d}{=}\psi\left(\zeta,\Theta\right)$.
Moreover, if there exists some $\zeta^{\prime}$ (independent
of $\Theta$) satisfying $\zeta^{\prime}\overset{d}{=}\psi\left(\zeta^{\prime},\Theta\right)$,
this calculation implies the equality 
$\zeta^{\prime}\overset{d}{=}\phi\left(\zeta^{\prime},\Theta\right)$.
By Theorem~\ref{thm:SecondPereq} we get
$\zeta^{\prime}\overset{d}{=}\zeta$,
as claimed.
\end{proof}

By Lemma~\ref{lem:perpCoalescence} and Kolmogorov's
consistency theorem, there exists a probability space supporting a sequence $\left\{ \Theta_{n}\right\} _{n\in\mathbb{Z}}$
of independent copies of $\Theta$ and a stationary Markov chain $\left\{ X_{n}\right\} _{n\in\mathbb{Z}}$,
satisfying $X_{n+1}=\psi\left(X_{n},\Theta_{n}\right)$ for all $n\in\mathbb{Z}$. 
In the remainder of the paper, $\{(X_n,\Theta_n)\}_{n\in\mathbb{Z}}$ denotes the corresponding Markov chain on
$(0,\infty)\times\mathcal{A}$.
In order to detect coalescence in Algorithm~\ref{alg:Algorithm 1},
we now construct a dominating process $\{D_n\}_{n\in\mathbb{Z}}$.

With this in mind, fix constants $\delta$ and $d$ satisfying 
$0<\delta<d<\frac{1}{\alpha\rho}$. \label{page:d_delta}
Let $I_{k}^{n}=1_{\left\{ S_{k}>e^{\delta\left(n-1-k\right)}\right\}}$
for all $n\in\mathbb{Z}$, $k\in\mathcal{Z}^{n}$ (see~\eqref{eq:def_mathcal_Z} above), 
where $S_k\sim S^{+}(\alpha,\rho)$ is the first component of $\Theta_k$ 
(see the first paragraph of Section~\ref{sec:SimulationDifficultProcesses}).
Fix $\gamma>0$ such that $\mathbb{E} S_{1}^{\gamma}<\infty$ (see~\eqref{eq:MellinStable}).
Markov's inequality implies 
\begin{equation}
p\left(m\right)=\mathbb{P}\left(S_{1}\leq e^{\delta m}\right)\geq1-e^{-\delta\gamma m}\mathbb{E} S_{1}^{\gamma},\quad m\geq0,\label{eq:probTails}
\end{equation}
and hence 
$\sum_{m=0}^{\infty}(1-p\left(m\right))<\infty$. Since $\{S_k\}_{k\in\mathbb{Z}}$ are independent, the Borel-Cantelli
lemma ensures that, for a fixed $n\in\mathbb{Z}$, 
the events 
$\left\{ S_{k}>e^{\delta\left(n-1-k\right)}\right\} =\left\{ I_{k}^{n}=1\right\} $
occur for only finitely many $k\in\mathcal{Z}^{n}$ a.s.  
Let $\chi_{n}$ be the smallest time beyond which the indicators 
$I_{k}^{n}$ are all zero: 
\begin{equation}
\label{eq:chi}
\chi_{n}=(n-1)\wedge\inf\left\{ k\in\mathcal{Z}^{n}:I_{k}^{n}=1\right\},
\end{equation}
with convention $\inf\emptyset =\infty$.
Note that $-\infty<\chi_{n}\leq n-1$ holds a.s. for all $n\in\mathbb{Z}$. Since the integers are countable,
we have $n-1\geq \chi_{n}>-\infty$ for all $n\in\mathbb{Z}$ a.s. 

Define the i.i.d. sequence $\left\{ F_{n}\right\} _{n\in\mathbb{Z}}$
by
$
F_{n}=d+\frac{1}{\alpha}\log\left(\Lambda_{n}U_{n}\right)$,
where $U_n$ and $\Lambda_n$ are the second and fourth components of $\Theta_n$, respectively 
(see the first paragraph of Section~\ref{sec:SimulationDifficultProcesses}).
Note that $d-F_{n}$ 
has the same law as a sum of (random) geometrically many independent exponential random variables and is hence
exponentially distributed with mean
$\mathbb{E}[d-F_{n}]=\frac{1}{\alpha\rho}$. 
Let $C=\left\{ C_{n}\right\} _{n\in\mathbb{Z}}$ be a random walk defined 
by $C_{0}=0$ and 
\begin{equation}
\label{eq:RW}
C_{n+1}=C_{n}-F_{n},\quad\text{$n\in\mathbb{Z}$.} 
\end{equation}
Recall definition~\eqref{eq:def_mathcal_Z}
and
let $R=\left\{ R_{n}\right\} _{n\in\mathbb{Z}}$
be the reflected process of the walk $\left\{ C_{n}\right\} _{n\in\mathbb{Z}}$,
that is
\begin{equation}
\label{eq:R}
R_{n}=\sup_{k\in\mathcal{Z}^{n+1}}C_{k}-C_{n},\quad n\in\mathbb{Z}.
\end{equation}
For any $n\in\mathbb{Z}$, define  the following random variables
\begin{eqnarray}
D_{n} & = &
\exp\left(R_{n}\right)\left(\frac{e^{\left(d-\delta\right)\left(\chi_{n}-n\right)}}{1-e^{\delta-d}}+
\sum_{k\in\mathcal{Z}^n_{\chi_{n}}}e^{-\left(n-1-k\right)d}S_{k}\left(1-U_{k}\right)^{\frac{1}{\alpha}}\right),\label{eq:DominatingProcess}\\
D_{n}^{\prime} & = &
\exp\left(R_{n}\right)\left(\frac{1}{1-e^{\delta-d}}+D_{n}^{\prime\prime}\right),
\qquad\text{where}\qquad D_{n}^{\prime\prime}  =  \sum_{k\in\mathcal{Z}^n}e^{-\left(n-1-k\right)d}S_{k}.
\label{eq:DoubleDominatingProcess}
\end{eqnarray}
The sum in~\eqref{eq:DominatingProcess} is taken to be zero if $\mathcal{Z}^n_{\chi_n}=\emptyset$, i.e. if $\chi_n=n$. 
Note that the series in $D_n''$ is absolutely convergent by the Borel-Cantelli lemma, but 
$D_n'$ cannot be simulated directly as it depends on an infinite sum. 
Finally,
define the random element 
$\Xi_n=\left(\Theta_{n},R_{n},D_{n}^{\prime}\right)$
for any $n\in\mathbb{Z}$.

\begin{lem}
\label{lem:recurrenceDomination}
{\normalfont(a)}  $X_{n}\leq D_{n}\leq D_{n}^{\prime}$ for all  
$n\in\mathbb{Z}$ a.s.\\
{\normalfont(b)} The processes $R=\left\{ R_{n}\right\}_{n\in\mathbb{Z}} $ and $\Xi=\{\Xi_n\} _{n\in\mathbb{Z}}$ 
are Markov, stationary and $\varphi$-irreducible (see definition~\citep[p.~82]{MR2509253})
with respect to the respective invariant distributions.
\end{lem}

\begin{proof}
(a) Since $\mathbb{E} F_{1}<0$, by the strong law of large numbers we have
$C_{-n}\to-\infty$ a.s. as $n\to\infty$. Hence  
$R_{n}<\infty$ for all $n\in\mathbb{Z}$ a.s. 
and a direct termwise comparison yields
$D_{n}^{\prime}\geq D_{n}$ for all $n\in\mathbb{Z}$.
It remains to prove that $X_{n}\leq D_{n}$ for all $n\in\mathbb{Z}$. 

Recall that the function 
$\theta\mapsto a(\theta)$ is defined in~\eqref{eq:varphi_and_a}.
Let $\tau_{n}=\sup\left\{ k\in\mathcal{Z}^{n}:X_{k}\leq a\left(\Theta_{k}\right)\right\}$
(with convention $\sup\emptyset = -\infty$)
be the last time the coalescence occurred before $n\in\mathbb{Z}$. 
If 
$\tau_{n}>-\infty$,
the value
$X_{1+\tau_{n}}$
does not depend on
$X_{\tau_{n}}$, 
and neither do the values of the chain taken at subsequent times. In particular, 
\[
X_{n}=\psi\left(X_{n-1},\Theta_{n-1}\right)=\underbrace{\psi\bigg(\cdots\psi}_{n-1-\tau_{n}}\bigg(W_{\tau_{n}}^{\frac{1}{\alpha\rho}}\left(1-U_{\tau_{n}}\right)^{\frac{1}{\alpha}}S_{\tau_{n}},\Theta_{\tau_{n}+1}\bigg),\cdots,\Theta_{n-1}\bigg).
\]
In general, 
by~\eqref{eq:varphi_and_A} and~\eqref{eq:def_mathcal_Z}, $X_n$ can be expressed as 
\begin{eqnarray}
\label{eq:X_n_Expression}
X_{n} & = &
\sum_{k\in\mathcal{Z}^n_{\tau_{n}+1}}\exp\left(\frac{1}{\alpha}
\sum_{j\in\mathcal{Z}^n_{k+1}}\log\left(\Lambda_{j}U_{j}\right)\right)\Lambda_{k}^{\frac{1}{\alpha}}\left(1-U_{k}\right)^{\frac{1}{\alpha}}S_{k}\\
\nonumber
& + &
1_{\{\tau_{n}>-\infty\}}\exp\left(\frac{1}{\alpha}\sum_{j\in\mathcal{Z}^n_{\tau_{n}+1}}\log\left(\Lambda_{j}U_{j}\right)\right)W_{\tau_{n}}^{\frac{1}{\alpha\rho}}\left(1-U_{\tau_{n}}\right)^{\frac{1}{\alpha}}S_{\tau_{n}},
\end{eqnarray}
where
sums over empty sets in~\eqref{eq:X_n_Expression} are defined to be equal to zero 
and, if $\tau_{n}=-\infty$, 
we define
$\mathcal{Z}^n_{\tau_{n}+1} =\mathcal{Z}^n$.
A termwise comparison then yields
\begin{eqnarray}
\nonumber
X_{n} & \leq & \sum_{k\in\mathcal{Z}^n}e^{C_{k+1}-C_{n}-\left(n-1-k\right)d}\left(1-U_{k}\right)^{\frac{1}{\alpha}}S_{k}\\
 & \leq & e^{R_{n}}\sum_{k\in\mathcal{Z}^n}e^{-\left(n-1-k\right)d}\left(1-U_{k}\right)^{\frac{1}{\alpha}}S_{k}\quad\text{for all \ensuremath{n\in\mathbb{Z}} a.s.}
\label{eq:First_bound_X_n}
\end{eqnarray}
Recall that $S_{k}\left(1-I_{k}^{n}\right)\leq e^{\delta\left(n-1-k\right)}\left(1-I_{k}^{n}\right)$
for all $k\in\mathcal{Z}^n$.
Since $I_{k}^{n}=0$ for $k<\chi_{n}$, we get 
\begin{eqnarray}
\nonumber
  \sum_{k\in\mathcal{Z}^n}e^{-\left(n-1-k\right)d}\left(1-U_{k}\right)^{\frac{1}{\alpha}}S_{k}
 & \leq & \sum_{k\in\mathcal{Z}^{\chi_n}}e^{-\left(n-1-k\right)\left(d-\delta\right)}\left(1-U_{k}\right)^{\frac{1}{\alpha}}
+\sum_{k\in\mathcal{Z}^n_{\chi_{n}}}e^{-\left(n-1-k\right)d}\left(1-U_{k}\right)^{\frac{1}{\alpha}}S_{k}\\
 & \leq & \frac{e^{\left(\chi_{n}-n\right)\left(d-\delta\right)}}{1-e^{\delta-d}}+\sum_{k\in\mathcal{Z}^n_{\chi_{n}}}e^{-\left(n-1-k\right)d}\left(1-U_{k}\right)^{\frac{1}{\alpha}}S_{k}.
\label{eq:Second_bound_X_n}
\end{eqnarray}
The inequalities in~\eqref{eq:First_bound_X_n}--\eqref{eq:Second_bound_X_n} and the definition in~\eqref{eq:DominatingProcess}
imply
$X_{n}\leq D_{n}$ for all $n\in\mathbb{Z}$
a.s.

\noindent (b) Note that $C_k-C_n=\sum_{i=k}^{n-1}F_i$ for all $k\in\mathcal{Z}^n$. Hence $R_n=\sup\{C_k-C_n:k\in\mathcal{Z}^{n+1}\}$ and $F_n$
are independent and the Markov property for $\left\{ R_{n}\right\}_{n\in\mathbb{Z}}$ follows from 
\[
R_{n}=\max\left\{ \sup_{k\in\mathcal{Z}^{n}}C_{k}-C_{n},0\right\} =\max\left\{ R_{n-1}+F_{n-1},0\right\} .
\]

By~\eqref{eq:DoubleDominatingProcess} we have
$D_{n}^{\prime\prime}=S_{n-1}+e^{-d}D_{n-1}^{\prime\prime}$.
Hence the pair $\left(R_{n},D_{n}^{\prime}\right)$ is a function
of the vector $\Xi_{n-1}=(\Theta_{n-1},R_{n-1},D_{n-1}^{\prime})$ 
(recall that $S_{n-1}$ is the first component of the random vector 
$\Theta_{n-1}$). Since the random elements $\Xi_{n-1}$ and $\Theta_n$ 
are independent, the process $\left\{\Xi_n\right\}_{n\in\mathbb{Z}}$
is Markov.  

The vector $\Xi_n=\left(\Theta_{n},R_{n},D_{n}^{\prime}\right)$
is in a bijective correspondence with $\left(\Theta_{n},R_{n},D_{n}^{\prime\prime}\right)$.

Since $\left\{ \Theta_{n}\right\}_{n\in\mathbb{Z}}$ are i.i.d., 
the following equality in law holds
\begin{eqnarray*}
\left(R_{n+1},D_{n+1}^{\prime\prime}\right)  =  
  \left(\sup_{j\in\mathcal{Z}^1}\sum_{k\in\mathcal{Z}^1_j} F_{n+k} ,\sum_{k\in\mathcal{Z}^1}e^{kd}S_{n+k}\right)
 & \overset{d}{=} & \left(\sup_{j\in\mathcal{Z}^1}\sum_{k\in\mathcal{Z}^1_j} F_{k} ,\sum_{k\in\mathcal{Z}^1}e^{kd}S_{k}\right),
\end{eqnarray*}
implying the stationarity of  $\left\{ \left(\Theta_{n},R_{n},D_{n}^{\prime\prime}\right)\right\}_{n\in\mathbb{Z}}$
and hence of $R$ and $\Xi$.

The process 
$R$ 
can jump to $0$ in a single step and has positive jumps of size at most $1/(\alpha\rho)-d$, both
with positive probability. Hence it will hit any subinterval of its state space $[0,\infty)$ from any starting point in a finite number of steps 
with positive probability, making it $\varphi$-irreducible~\citep[p.~82]{MR2509253} with respect to its invariant law.

Since $\Theta_n$ is independent of 
$(R_n,D_n'')$, 
the $\varphi$-irreducibility of $\left\{\Xi_{n}\right\}_{n\in\mathbb{Z}}$ follows 
if, starting from an arbitrary point, we can prove that the process 
$\{(R_n,D_n'')\}_{n\in\mathbb{Z}}$ 
hits any rectangle in the product $[0,\infty)\times (0,\infty)$
with positive probability. 
Since we already know that   
$R$ hits intervals 
and has (arbitrarily) small positive jumps with positive probability, 
the independence of 
$\{D''_n\}_{n\in\mathbb{Z}}$
and 
$R$,
together with the fact that 
$D_n''$ has a positive density, imply the final statement of the lemma.
\end{proof}

\begin{proof}[Proof of Theorem~\ref{thm:FiniteExpectedTerminationTime}]
\label{proof:FiniteExpectedTerminationTime} 
By Lemma~\ref{lem:recurrenceDomination}(ii), 
$\Xi$
is
$\pi$-irreducible, 
where
$\pi$ 
denotes the invariant law of $\Xi$. 
Hence, by~\citep[Prop.~10.1.1]{MR2509253}, 
$\Xi$ is recurrent, meaning that the expected number of visits of the chain 
$\Xi$ to any set charged by $\pi$ is infinite for all starting points. 
By~\citep[Thm~9.0.1]{MR2509253}, the chain $\Xi$ is Harris recurrent on a complement
of a $\pi$-null set. Put differently, for any starting point, the number of visits $\Xi$ makes to any set charged by
$\pi$ is infinite almost surely. 

Consider the Markov chain $\Psi=\{\Psi_n\}_{n\in\mathbb{N}}$, where $\mathbb{N}=\{0,1,\ldots\}$ and
$\Psi_n=\Xi_{-n}$.
In the language of~\cite{MR758799}, $\Psi$
is a chain dual to $\Xi$ with respect to $\pi$. 
In particular, the invariant law of $\Psi$ equals $\pi$.
Since $\Xi$ is Harris recurrent on a state space with a countably generated
$\sigma$-algebra, \citep[Thm~8.1.1]{MR758799} implies that there exists a modification of $\Psi$
(again denoted by $\Psi$) that is also Harris recurrent.
Since 
$\mathbb{P}\left(a\left(\Theta_{-n}\right)\geq D_{-n}^{\prime}\right)>0$
for any $n\in\mathbb{N}$,
it follows that the $\Psi$-stopping time  
$\sigma^{\prime}=\inf\left\{ n>0:a\left(\Theta_{-n}\right)\geq D_{-n}^{\prime}\right\}$
is finite almost surely. Moreover, by~\citep[Thm~11.1.4]{MR2509253} we have $\mathbb{E} [\sigma'|\Psi_0]<\infty$ almost surely.

Recall that $\sigma=\inf\left\{ n>0:a\left(\Theta_{-n}\right)\geq D_{-n}\right\}$ 
is the number of steps taken backwards in time in Algorithm~\ref{alg:Algorithm 1}. 
By Lemma~\ref{lem:recurrenceDomination}(i) we have $\sigma\leq \sigma'$. Since, by definition $\Psi_0=\Xi_0$, the claim follows. 
\end{proof}

\section{Backward Simulation of $\left\{ \left(D_{n},\Theta_{n}\right)\right\} _{n\in\mathbb{Z}}$}
\label{sec:Backward_Sim}

A key step in Algorithm~\ref{alg:Algorithm 1} consists of simulating the process 
$\{(D_n,\Theta_n)\}_{n\in\mathbb{Z}}$ backwards in time until the random time
$\sigma=\inf\left\{ n>0:a(\Theta_{-n})\geq D_{-n}\right\}$ 
(see~\eqref{eq:def_mathcal_Z} and~\eqref{eq:varphi_and_a} for the definitions of 
$\mathcal{Z}^1$ and $a(\theta)$, respectively). 
The forthcoming Algorithm 2 is responsible for this step. Recall that 
$\{\Theta_n\}_{n\in\mathbb{Z}}$ is an i.i.d. sequence with 
$\Theta_n=(S_n,U_n,W_n,\Lambda_n)$ having independent components, 
where $S_n$, $U_n$ and $\Lambda_n$ are distributed as in 
Theorem~\ref{thm:SecondPereq} and $W_n\sim U(0,1)$.

At time 
$n\in\mathbb{Z}$,
the dominating process $D$ in~\eqref{eq:DominatingProcess}
depends on three components:  the sequence 
$\left(\chi_{n},\left\{ S_{k}\right\} _{k\in\mathcal{Z}_{\chi_{n}}^{0}}\right)$, 
 the all-time maximum $\sup_{k\in\mathcal{Z}^{n+1}} \{C_k\}$ and $C_n$ (via the reflected process $R$, see~\eqref{eq:RW}-\eqref{eq:R})
and  the uniform random variables  $\left\{ U_{k}\right\} _{k\in\mathcal{Z}_{\chi_{n}}^{0}}$. 
The time $\chi_n$ in~\eqref{eq:chi} is the last time before $n$ the random variables $\left\{ S_{k}\right\}_{k\in\mathcal{Z}^{0}}$
exceed a certain adaptive exponential bound.  
Algorithm~\ref{alg:Algorithm 3} for sampling 
$\left(\chi_{n},\left\{ S_{k}\right\} _{k\in\mathcal{Z}_{\chi_{n}}^{0}}\right)$ 
is given in Section~\ref{subsec:StableInc} below.
A sample for $(R_n,C_n)$ requires the joint forward simulation of the dual random walk $-C$ and its ultimate maximum. 
This problem was solved in~\citep{MR2865624}.
The algorithm in~\citep{MR2865624},
stated for completeness
as Algorithm~\ref{alg:Algorithm 5}
of
Section~\ref{subsec:BlanchetSigman} below 
for the random walk $C$ in~\eqref{eq:RW}, 
requires the simulation of the walk under the
exponential change of measure. 

Since the increments of $C$ are shifted negative exponential random variables under the original measure, 
they remain in the same class under the exponential change of measure, making the simulation in 
Algorithm~\ref{alg:Algorithm 5} simple. 
Finally, heaving simulated $(R,C)$ backwards in time, we need to recover the random variables 
$\Lambda_{k}$ and $U_{k}$,
conditional on the values of increments $F_k=d+(1/\alpha) \log(U_n \Lambda_n)$ we have observed. 
Algorithm~\ref{alg:Algorithm 2-1}
in Section~\ref{subsec:ULambda_given_F} below describes this step.

\begin{algorithm}
\caption{Backward simulation of $\left(\sigma,\left\{ \left(D_{n},\Theta_{n}\right)\right\} _{n\in\mathcal{Z}_{-\sigma}^{0}}\right)$}
\label{alg:Algorithm 2} 
\begin{algorithmic}[1]
	\State{Sample $\chi_{-1}$ and $\{S_k\}_{k\in \mathcal{Z}^0_{\chi_{-1}}}$}
		\Comment{Algorithm~\ref{alg:Algorithm 3}}
	\State{Sample $\{(R_k,C_k,\Lambda_k,U_k)\}_{k\in \mathcal{Z}^0_{N_{-1}}}$ for some $N_{-1}\leq \chi_{-1}$}
			\Comment{Algorithms~\ref{alg:Algorithm 5} \& \ref{alg:Algorithm 2-1}}
	\State{Bundle up $\{\Theta_k\}_{k\in \mathcal{Z}^0_{\chi_{-1}}}$ and compute $D_{-1}$}
	\State{Put $n:=-1$}
	\While{$D_n>a(\Theta_n)$}
		\State{Put $n:=n-1$}
		\State{Sample $\chi_{n}$ and $\{S_k\}_{k\in \mathcal{Z}^{\chi_{n+1}}_{\chi_n}}$ conditional on $(\chi_{n+1},\{S_k\}_{k\in \mathcal{Z}^{\chi_{n+1}}_{\chi_n}})$}
		\Comment{Algorithm~\ref{alg:Algorithm 3}}
		\State{Sample $\{(R_k,C_k,\Lambda_k,U_k)\}_{k\in \mathcal{Z}^{N_{n+1}}_{N_n}}$ for some $N_n\leq \chi_n$}
			\Comment{Algorithms~\ref{alg:Algorithm 5} \& \ref{alg:Algorithm 2-1}}
		\State{Bundle up $\{\Theta_k\}_{k\in \mathcal{Z}^{\chi_{n+1}}_{\chi_n}}$, and compute $D_n$}
	\EndWhile
	\State{Put $\sigma = -n$}
	\State{\Return $(\sigma,\{ \Theta_{k}\} _{k\in \mathcal{Z}^0_{-\sigma}})$}
\end{algorithmic}
\end{algorithm}
The number of steps $N_{-1}$ (resp. $N_{n}$) in line 2 (resp. 8) of Algorithm~\ref{alg:Algorithm 2}
is random since Algorithm~\ref{alg:Algorithm 5}, which outputs the all-time maximum of the random walk,
may need more values of the random walk  than required to recover the previous value of the dominating process $D_{-1}$
(resp. $D_{n}$).\footnote{In the notation of Section~\ref{subsec:BlanchetSigman} below, the integers $N_{n}$ take the form $\Delta\left(\tau_{m}\right)$.}
The running time of Algorithm~\ref{alg:Algorithm 3} is random but has moments of all orders (see Lemma~\ref{lem:finiteTerminationSubroutines}
in Section~\ref{subsec:StableInc} below). Algorithm~\ref{alg:Algorithm 2-1} executes a loop of length equal to the number of steps in 
the random walk $C$ the algorithm is applied to, with each step sampling one Poisson and one 
Beta random variables (see Section~\ref{subsec:ULambda_given_F}
below). Hence both Algorithms~\ref{alg:Algorithm 3} and~\ref{alg:Algorithm 2-1} are fast (see Section~\ref{sec:numerics}). 
Algorithm~\ref{alg:Algorithm 5} of~\citep{MR2865624} (see Section~\ref{subsec:BlanchetSigman} below) 
runs sequentially Algorithms~\ref{alg:SubAlgorithm 5-1},~\ref{alg:SubAlgorithm 5-2} 
and~\ref{alg:SubAlgorithm 5-3}. Each of these algorithms is reliant on rejection sampling and 
has a finite expected running time, which is easy to quantify in terms  of the increments of the walk $C$. 

\subsection{\label{subsec:StableInc}Simulation of $\left(\chi_{n},\left\{ S_{k}\right\} _{k\in\mathcal{Z}_{\chi_{n}}^{0}}\right)$}
Consider independent Bernoulli random variables $\{J_{n}\} _{n=1}^\infty$
with computable $p_{n}=\mathbb{P}(J_{n}=0)$, $n\geq 1$, satisfying
$\sum_{n=1}^\infty(1-p_{n})<\infty$.
By the Borel-Cantelli Lemma the random time
$\tau=\sup\{n\geq0:J_n=1\}^+$ (with convention $\sup\emptyset =-\infty$)
satisfies $\tau\in\mathbb{N}$ a.s. Clearly, $J_{n}=0$ for all $n>\tau$, 
and $\{\tau<n\} =\bigcap_{k=n}^{\infty}\{J_{k}=0\}$
implies $\mathbb{P}(\tau<n)=\prod_{k=n}^{\infty}p_{k}$.
If there exists $n^{\ast}\geq1$ such that for all $n\geq n^{\ast}$
we have a positive computable lower bound $q_{n}\leq\prod_{k=n}^{\infty}p_{k}$,
then we can simulate $(\tau,\{J_k\}_{k\in\{0,\ldots,\tau\}})$ as follows. 

Define the auxiliary function $F:\left(0,1\right)\times\left(0,1\right)\to\left\{ 0,1\right\} \times\left(0,1\right)$
by the formula
\[
F\left(u,p\right)=\begin{cases}
\left(0,\frac{u}{p}\right) & \text{ if }u\leq p,\\
\left(1,\frac{u-p}{1-p}\right) & \text{ if }u>p.
\end{cases}
\]
The following observation is simple but crucial: 
for any $p\in(0,1)$ and $U\sim U\left(0,1\right)$,
the components of the vector 
$\left(J,V\right)=F\left(U,p\right)$ 
are independent, $J$ is Bernoulli with 
$\mathbb{P}(J=0)=p$ and $V\sim U\left(0,1\right)$. 

Sample $\left\{ J_{n}\right\} _{n\in\mathcal{Z}_1^{n^\ast}}$ and
an independent $U^{\left(n^{\ast}\right)}\sim U\left(0,1\right)$.
Let $\left(J_{n^\ast},U^{\left(n^\ast+1\right)}\right)=F\left(U^{\left(n^\ast\right)},p_{n^\ast}\right)$.
Hence
$J_{n^{\ast}}$ has the correct distribution and is independent of 
$U^{\left(n^{\ast}+1\right)}\sim U(0,1)$.
Thus,
$J_{n^{\ast}}$ is independent of 
$F\left(U^{\left(n^{\ast}+1\right)},p_{n^{\ast}+1}\right)=\left(J_{n^{\ast}+1},U^{\left(n^{\ast}+2\right)}\right)$.
Define recursively
$\left(J_{n},U^{\left(n+1\right)}\right)=F\left(U^{\left(n\right)},p_{n}\right)$
for $n\geq n^\ast+2$ and note that 
the sequence $\left\{ J_{n}\right\}_{n\in\mathbb{N}}$
of Bernoulli random variables is i.i.d. 
Moreover, the sequence $\{U^{(n)}\}_{n\geq n^\ast}$ 
detects the value of $\tau$ since  
$\left\{ U^{\left(n\right)}\leq q_{n}\right\} 
\subseteq \left\{ U^{\left(n\right)}\leq\prod_{k=n}^{\infty}p_{k}\right\} 
=\left\{ \tau<n\right\}$.

\begin{algorithm}
	\caption{Simulation of $(\tau,\{J_k\}_{k\in\{1,\ldots,\tau\}})$}\label{alg:Algorithm 3} 
\begin{algorithmic}[1]
	\State{Sample $J_1,\ldots, J_{n^\ast-1}$ and put $n:=n^\ast-1$}
	\State{Sample $U\sim U(0,1)$}
	\Loop
		\State{Put $n := n+1$}
		\If{$U>p_n$}
			\State{Put $J_n:=1$ and update $U:=\frac{U-p_n}{1-p_n}$}
		\ElsIf{$U \leq q_n$}
			\State{Compute $\tau$ from $J_1,\ldots,J_{n-1}$ and \bf{exit loop}}
		\Else
			\State{Put $J_n:=0$ and update $U:=\frac{U}{p_n}$}
		\EndIf
	\EndLoop
	\State{\Return $(\tau,\{J_k\}_{k\in\{1,\ldots,\tau\}})$}
\end{algorithmic}\end{algorithm}

Algorithm~\ref{alg:Algorithm 3} samples a single uniform random variable and performs a binary search. 
Its running time 
$\varsigma=\inf\left\{ n\geq n^{\ast}:U^{\left(n\right)}\leq q_{n}\right\}\geq\tau+1$ 
(with convention $\inf\emptyset =\infty$)
has the following properties.
\begin{lem}
\label{lem:finiteTerminationSubroutines}
{\normalfont(a)} If $\lim_{n\to\infty}q_{n}=1$ then $\mathbb{P}(\varsigma<\infty)=1$.\\
{\normalfont(b)} If $\sum_{n=n^{\ast}}^{\infty}\left(1-q_{n}\right)<\infty$ then
$\mathbb{E} \varsigma<\infty$.\\
{\normalfont(c)} If $\sum_{n=n^{\ast}}^{\infty}\left(1-q_{n}\right)e^{tn}<\infty$
for some $t>0$, then $\mathbb{E} e^{t\varsigma}<\infty$.\\
{\normalfont(d)} If $q_n p_{n-1}\geq q_{n-1}$ for $n>n^\ast$, then the converses of (a), (b) and (c) are also true.
\end{lem}
\begin{rem}
	At the cost of additional operations, one may always 
	construct a sequence $\{q_n^\prime\}_{n=n^\ast}^\infty$ that 
	satisfies (d). Indeed, let $q_{n^\ast}^\prime=q_{n^\ast}$ and 
	define recursively 
	$q_n^\prime = \max\{q_n,q^\prime_{n-1}/p_{n-1}\}$ for $n> n^\ast$,
	then these satisfy condition (d), are computable and inductively satisfy $q_{n}^\prime\leq \prod_{k=n}^\infty p_k$ for $n\geq n^\ast$.
	This consideration shows that our conditions are sharp.
\end{rem}
\begin{proof}
(a) For all $n\geq n^{\ast}$ we have $\left\{ \varsigma\leq n\right\} 
\supseteq\left\{ U^{\left(n\right)}\leq q_{n}\right\} $,
then $\mathbb{P}(\varsigma>n)\leq\mathbb{P}(U^{(n)}>q_n)=1-q_{n}$.
Hence $\mathbb{P}(\varsigma=\infty)=\lim_{n\to\infty}\mathbb{P}(\varsigma>n)
\leq\lim_{n\to\infty}(1-q_{n})=0$
and the sufficiency follows.
\\(b) Similarly, 
$\mathbb{E} \varsigma=\sum_{n=0}^{\infty}\mathbb{P}\left(\varsigma>n\right)
\leq n^{\ast}+\sum_{n=n^{\ast}}^{\infty}\left(1-q_{n}\right)<\infty$
and the claim follows.
\\(c) Note that $\left(e^{t}-1\right)\sum_{m=0}^{n-1}e^{tm}=e^{tn}-1$.
Exchanging the order of summation in the third equality of the following estimate
implies
(c): 
\begin{eqnarray*}
\mathbb{E} e^{t\varsigma} & = & \sum_{n=0}^{\infty}\mathbb{P}\left(\varsigma=n\right)e^{tn}  
=  \sum_{n=0}^{\infty}\mathbb{P}\left(\varsigma=n\right)\left(1+\left(e^{t}-1\right)\sum_{m=0}^{n-1}e^{tm}\right)\\
 & = & 1+\left(e^{t}-1\right)\sum_{m=0}^{\infty}e^{tm}\mathbb{P}\left(\varsigma>m\right)
  \leq  e^{tn^{\ast}}+\left(e^{t}-1\right)\sum_{n=n^{\ast}}^{\infty}\left(1-q_{n}\right)e^{tn}<\infty.
\end{eqnarray*}
(d) Condition (d) and the relation 
$(\tau+1)\vee n^\ast=\inf\{k\geq n^\ast:U^{(k)}\leq \prod_{j=k}^\infty p_j\}$
imply for $n\geq k\geq n^\ast$,
\[
\{(\tau+1)\vee n^\ast = k, \varsigma\leq n\} =
\begin{cases}
\left\{ U^{(k-1)}\in 
\bigg[ p_{k-1},p_{k-1}+(1-p_{k-1})q_n
\prod_{j\in\mathcal{Z}_{k}^n}p_j\bigg]\right\} & k>n^\ast,\\
\left\{ U^{(n^\ast)}\in \bigg[ 0,q_n\prod_{j\in\mathcal{Z}_{n^\ast}^n}p_j\bigg]\right\}& k=n^\ast.
\end{cases}
\]
Thus, a simple calculation yields
\[
\mathbb{P}(\varsigma \leq n)
=q_n\prod_{j\in\mathcal{Z}_{n^\ast}^n}p_j+\sum_{k\in\mathcal{Z}_{n^\ast}^{n}}q_n(1-p_k)\prod_{j\in\mathcal{Z}_{k+1}^n}p_j
=q_{n},
\]
and the result follows from standard probability theory.
\end{proof}

In Algorithm~\ref{alg:Algorithm 2} we are required to sample 
$\left(\chi_{0},\left\{ S_{k}\right\} _{\mathcal{Z}_{\chi_{0}}^{0}}\right)$,
and then, iteratively for $n\in\mathcal{Z}^0$, $\chi_{n}$ and the remaining 
$\left\{ S_{k}\right\} _{k\in\mathcal{Z}_{\chi_{n}}^{\chi_{n+1}}}$,
given the known values $\left(\chi_{n+1},\left\{ S_{k}\right\} _{k\in\mathcal{Z}_{\chi_{n+1}}^{0}}\right)$.
To apply Algorithm~\ref{alg:Algorithm 3}, we need a computable lower bound 
on the product of probabilities $p\left(m\right)=\mathbb{P}(S_1\leq e^{\delta m})$, $m\in\mathbb{N}$. 
Recall the exponential lower bound on 
$p\left(m\right)$ in~\eqref{eq:probTails} and define
$m^{\ast}=\left\lfloor \frac{1}{\delta\gamma}\log\mathbb{E} S_1^{\gamma}\right\rfloor^+ +1$
(here $\lfloor x\rfloor=\sup\{n\in\mathbb{Z}:n\leq x\}$ for any $x\in\mathbb{R}$).
Note that for any $m\geq m^{\ast}$ we have $e^{-\delta\gamma m}\mathbb{E} S_1^{\gamma}<1$
and may hence define 
$\overline{p}\left(m\right) = 
 \exp\left(-\frac{1}{1-e^{-\delta\gamma}}\frac{e^{-\delta\gamma m}
 \mathbb{E} S_{1}^{\gamma}}{1-e^{-\delta\gamma m}\mathbb{E} S_{1}^{\gamma}}\right)\in (0,1)$.
The inequality in~\eqref{eq:probTails} implies
\begin{eqnarray*}
\prod_{j=m}^{\infty}p\left(j\right) 
& \geq & \prod_{j=m}^{\infty}\left(1-e^{-\delta\gamma j}\mathbb{E} S_{1}^{\gamma}\right)
=\exp\left(\sum_{j=m}^{\infty}\log\left(1-e^{-\delta\gamma j}\mathbb{E} S_{1}^{\gamma}\right)\right)\\
 & = &
\exp\left(-\sum_{j=m}^{\infty}\sum_{k=1}^{\infty}\frac{1}{k}e^{-\delta\gamma
jk}\left(\mathbb{E} S_{1}^{\gamma}\right)^{k}\right)\geq\exp\left(-\sum_{k=1}^{\infty}\frac{e^{-\delta\gamma
mk}\left(\mathbb{E} S_{1}^{\gamma}\right)^{k}}{1-e^{-\delta\gamma
k}}\right)\geq  \overline{p}\left(m\right).
\end{eqnarray*}
Since for any $k\in\mathcal{Z}^{0}_{\chi_0}$ we have 
$\mathbb{P}(I^0_k=0) = \mathbb{P}(S_k\leq e^{-(k+1)\delta})=p(-(k+1))$,
Algorithm~\ref{alg:Algorithm 3} can be applied (with $n^\ast=m^\ast$) to sample the sequence 
$\left\{ I_{k}^{0}\right\}_{k\in\mathcal{Z}^{0}_{\chi_0}}$. 
Moreover, for $m\in\mathbb{N}$ we get 
\begin{eqnarray*}
\overline{p}\left(m^{\ast}+m\right) & \geq & \exp\left(-re^{-\delta\gamma m}\right)\geq1-re^{-\delta\gamma m},
\quad
\text{where}\quad r  =  \frac{e^{-\delta\gamma m^{\ast}}\mathbb{E} S_{1}^{\gamma}}{\left(1-e^{-\delta\gamma}\right)\left(1-e^{-\delta\gamma m^{\ast}}\mathbb{E} S_{1}^{\gamma}\right)}>0.
\end{eqnarray*}
Hence, 
for any $t\in(0,\delta\gamma)$,
Lemma~\ref{lem:finiteTerminationSubroutines}(c) implies that the running time $\varsigma$ satisfies $\mathbb{E} [e^{\varsigma t}]<\infty$
and therefore possesses moments of all orders. 
Having obtained $\left(\chi_{0},\left\{ I_{k}^{0}\right\}_{k\in\mathcal{Z}^{0}_{\chi_0}}\right)$,
for $k\in\mathcal{Z}^{0}_{\chi_0}$, we  
sample $S_{k}$ as $S^{+}\left(\alpha,\rho\right)$ conditional on 
$S_{k}\leq e^{-\delta\left(k+1\right)}$
(if $I_{k}^{0}=0$) or $S_{k}>e^{-\delta\left(k+1\right)}$ (if
$I_{k}^{0}=1$),
yielding a sample of 
$\left(\chi_{n},\left\{ S_{k}\right\} _{k\in\mathcal{Z}_{\chi_{n}}^{0}}\right)$.

Assume now that we have already sampled 
$\left(\chi_{n+1},\left\{ S_{k}\right\} _{k\in\mathcal{Z}_{\chi_{n+1}}^{0}}\right)$.
The adaptive exponential bounds in the indicators $I^{n+1}_k$ and $I^n_k$ are different 
(see Figure~\ref{fig:adaptive_bound}) 
and the relevant probabilities take the form
\[
p^{\prime}\left(m\right)=\mathbb{P}\left(\left.S_{1}\leq e^{\delta m} \right|S_{1}\leq e^{\delta\left(m+1\right)}\right),\quad m\in\mathbb{N}.
\]
Since 
$\{S_{1}\leq e^{\delta m}\}
\subset 
\{S_{1}\leq e^{\delta (m+1)}\}$,
the inequality 
$p^{\prime}\left(m\right)\geq p\left(m\right)$
holds for any $m\in\mathbb{N}$. 
Thus
\begin{eqnarray*}
\prod_{j=m}^{\infty}p^{\prime}\left(j\right)\geq\prod_{j=m}^{\infty}p\left(j\right) & \geq & \overline{p}\left(m\right)
\end{eqnarray*}
and Algorithm~\ref{alg:Algorithm 3} can be applied with $n^\ast=\max\{m^\ast,n-\chi_{n+1}\}$.
The same argument as above shows that the running time $\varsigma$ has moments of all orders.

\begin{figure}[H]
\begin{centering}
\par\end{centering}
\centering{}\begin{tikzpicture} 
	\begin{axis} 
		[ymin=0,
		xlabel=$k$, 
		ylabel=$\{S_k\}$, 
		width=12cm,
		height=5cm,
		axis on top=true,
		ytick=\empty,
		xtick=\empty,
		xticklabels=none,
		yticklabels=none,
		axis x line=middle, 
		axis y line=middle, 
		legend pos=north east ] 

		\addplot [line width = 1, smooth, domain=-11.5:0] {exp(-.15*(x+1))};
		\addplot [dashed, line width = 1, smooth, domain=-11.5:0] {exp(-.15*(x+2))};
		\addplot [dotted, line width = 1, smooth, domain=-11.5:0] {exp(-.15*(x+3))};

		\node [fill=blue, circle, scale=0.5] at (axis cs: -2,0) {};
		\node [fill=blue, circle, scale=0.5] at (axis cs: -7,0) {};
		\node [fill=blue, circle, scale=0.5] at (axis cs: -10,0) {};
		\node [above right] at (axis cs: -2,0) {$\chi_0$};
		\node [above right] at (axis cs: -7,0) {$\chi_{-1}$};
		\node [above right] at (axis cs: -10,0) {$\chi_{-2}$};

		\node [fill=red, circle, scale=0.5] at (axis cs: -1,0.19) {};
		\node [fill=red, circle, scale=0.5] at (axis cs: -2,1.93) {}; 
		\node [fill=red, circle, scale=0.5] at (axis cs: -3,1.28) {};
		\node [fill=red, circle, scale=0.5] at (axis cs: -4,0.73) {};
		\node [fill=red, circle, scale=0.5] at (axis cs: -5,0.49) {};
		\node [fill=red, circle, scale=0.5] at (axis cs: -6,1.33) {};
		\node [fill=red, circle, scale=0.5] at (axis cs: -7,2.23) {}; 
		\node [fill=red, circle, scale=0.5] at (axis cs: -8,1.52) {};
		\node [fill=red, circle, scale=0.5] at (axis cs: -9,0.43) {};
		\node [fill=red, circle, scale=0.5] at (axis cs: -10,3.07) {};
		\node [fill=red, circle, scale=0.5] at (axis cs: -11,0.87) {};

		\node [below right] at (axis cs: -2,1.93) {$S_{-2}$};
		\node [above right] at (axis cs: -7,2.23) {$S_{-7}$};
		\node [below left] at (axis cs: -10,3.07) {$S_{-10}$};

		\legend {$k\mapsto e^{-\delta(k+1)}$,
				$k\mapsto  e^{-\delta(k+2)}$,
				$k\mapsto  e^{-\delta(k+3)}$};
	\end{axis}
\end{tikzpicture}
\caption{\footnotesize
	The adaptive exponential bounds $k\mapsto e^{\delta(n-k-1)}$ for $n\in\{0,-1,-2\}$ 
	with the corresponding stable random variables conditioned to be positive 
	$\{S_k\}_{k\in\mathcal{Z}^0}$ and the times $\{\chi_k\}_{k\in\mathcal{Z}^1}$ used for the 
	construction of the dominating process $\{D_n\}_{n\in\{0,-1,-2\}}$ in~\eqref{eq:DominatingProcess}.
}
\label{fig:adaptive_bound}
\end{figure}
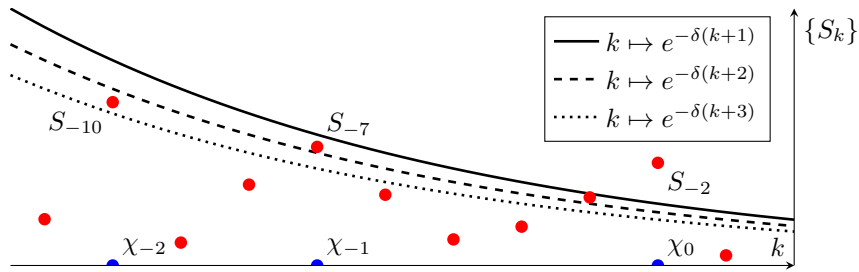

\subsection{\label{subsec:BlanchetSigman}Simulation of the Random Walk and Its Reflected Process from~\citep{MR2865624}}
In this section we present an overview of the algorithm in~\citep{MR2865624}  for the joint simulation of $(C,R)$
defined in~\eqref{eq:RW}-\eqref{eq:R}.
We refer to~\citep{MR2865624} and~\citep{MR1759244} for the proofs (the latter paper contains the simulation algorithm
for the ultimate maximum of a random walk with negative drift and provides a basis for the simulation algorithm in~\citep{MR2865624}).

Let $\eta=\eta(d)$ be the unique positive root of $\psi_{d}(\eta)=0$,
where $\psi_{d}(t)=\log(\mathbb{E} e^{tF_0})=dt-\log(1+t/(\alpha\rho))$.
Note that $\psi_{d}^{\prime}(\eta)=d-\frac{1}{\alpha\rho+\eta}>0$
and $\eta 
=-\alpha\rho-W_{-1}\left(-\alpha\rho de^{-\alpha\rho d}\right)/d$,
where $W_{-1}$ is the secondary branch of the Lambert W function.
Since $\mathbb{E}[\exp(\eta F_n)]=1$ for all $n\in\mathbb{Z}$,
the process $\{ \exp(\eta C_{n})\}_{n\in\mathcal{Z}^{1}}$
is a positive backward martingale started at one, 
thus inducing a probability measure
$\mathbb{P}^{\eta}$
on $\sigma$-algebras
$\sigma\left(C_{k};k\in\mathcal{Z}_{n}^{1}\right)$,
$n\in\mathcal{Z}^{1}$,
by the formula
$\mathbb{P}^{\eta}(A)=\mathbb{E}[1_{A}e^{\eta C_{n}}]$
where $A\in\sigma\left(C_{k};k\in\mathcal{Z}_{n}^{1}\right)$. 
Under $\mathbb{P}^{\eta}$, the process $C$ remains a random walk with i.i.d.
increments satisfying $(\alpha\rho+\eta)(d-F_{n})\sim Exp(1)$.
Hence $\mathbb{E}^{\eta}[C_{-1}]=\psi_{d}^{\prime}\left(\eta\right)>0$, implying
$\lim_{n\to-\infty}C_{n}=\infty$, $\mathbb{P}^\eta$-a.s. by the strong law of large numbers.

For any $k\in\mathbb{Z}$ define (with convention $\sup \emptyset = -\infty$)
\begin{equation}\label{eq:T_x^k}
T_{x}^{k}=\begin{cases}
\sup\left\{ n\in\mathcal{Z}^{k}:C_{n}-C_{k}>x\right\}  & \text{ if }x>0,\\
\sup\left\{ n\in\mathcal{Z}^{k}:C_{n}-C_{k}<x\right\}  & \text{ if }x<0.
\end{cases}
\end{equation}
For ease of notation we let
$T_{x}=T_{x}^{0}$. 
Let $E$ be an independent exponential random variable with mean one. Then, for $x>0$, we have 
$\mathbb{P}\left(R_{0}>x\right)=\mathbb{P}^{\eta}\left(L_{E/\eta}>x\right)$,
where $L_{x}=\inf\left\{ y\geq0:C_{T_{y}}>x\right\} $ is the right
inverse of $x\mapsto C_{T_x}$, 
see e.g.~\citep{MR1759244}.
Hence for $x\in(0,x')$, where $x'\leq\infty$, sampling  $1_{\left\{ R_{0}>x\right\} }=1_{\left\{ T_{x}>-\infty\right\} }$,
conditional on 
$1_{\left\{ R_{0}\leq x'\right\} }=1_{\left\{ T_{x'}=-\infty\right\}}$,
in finite time amounts to sampling $E$ and $C_{-1},\ldots,C_{T_{E/\eta}}$
under $\mathbb{P}^{\eta}$, see Algorithm~\ref{alg:SubAlgorithm 5-1} below.  

\begin{algorithm}
\caption{Simulation of $1_{\left\{ R_{0}>x\right\} }$
conditional on $\left\{ R_{0}\leq x^{\prime}\right\} $}
\label{alg:SubAlgorithm 5-1}
\begin{algorithmic}[1]
	\Require{$\infty \geq x^\prime >x >0$}
	\Loop
		\State{Sample $E\sim Exp(1)$}
		\If{$E/\eta\leq x$}
			\State{\Return 0}
		\Else
			\State{Sample $C_0=0,C_{-1},\ldots,C_{T_{E/\eta}}$ under $\mathbb{P}^\eta$}
			\State{Compute $L_{E/\eta}$}
			\If{$L_{E/\eta}\leq x^\prime$}\Comment{Accept sample}
				\State{\Return $1_{\{L_{E/\eta}>x\}}$}
			\EndIf
		\EndIf
	\EndLoop
\end{algorithmic}
\end{algorithm}
\begin{rem}
	Since $L_x\leq x$, then the condition $E/\eta\leq x$ implies $L_{E/\eta}\leq x$, 
thus identifying $1_{\{L_{E/\eta}>x\}}=0$ (see line~3) and saving the computational effort of running 
all subsequent lines. This algorithm repeats independent experiments 
with success probability $\mathbb{P}^\eta(L_{E/\eta}\leq x^\prime)>0$. The expected runtime of each 
iteration in the loop is bounded above by 
$(\eta^{-1}+d)/\psi'_d(\eta)$, see~\citep[Eq.~(2.3)]{MR1759244}.
Hence the expected running time of Algorithm~\ref{alg:SubAlgorithm 5-1} is finite.
\end{rem}

In Algorithm~\ref{alg:Algorithm 5} below 
we need to sample the path of the random walk $\{C_k\}_{k\in\mathcal{Z}^1_{T_x}}$ conditioned on the event $\left\{ R_{0}\in(x,x')\right\}$,
where $0<x<x'\leq\infty$. By a rejection sampling method under $\mathbb{P}^\eta$  
and Algorithm~\ref{alg:SubAlgorithm 5-1} (see~\citep[Lemma~3]{MR2865624}), this can be achieved as follows.

\begin{algorithm}
\caption{Simulation of $C_{0},\ldots,C_{T_{x}}$
conditional on $\left\{ T_{x}>-\infty=T_{x^{\prime}}\right\} $}
\label{alg:SubAlgorithm 5-2}
\begin{algorithmic}[1]
	\Require{$\infty\geq x^\prime >x >0$}
	\Loop
		\State{Sample $C_0=0, C_{-1}, \ldots, C_{T_{x}}$ under $\mathbb{P}^\eta$}
		\State{Given $C_{T_{x}}$, sample independent $1_{\{R_0 ^\prime\leq x^\prime-C_{T_x}\}}$ and $U\sim U(0,1)$}
			\Comment{Algorithm~\ref{alg:SubAlgorithm 5-1}} 
		\If{$U \leq \exp(-\eta C_{T_{x}})$ and $1_{\{R_0 ^\prime\leq x^\prime-C_{T_x}\}}=1$}
			\Comment{Accept sample}
			\State{\Return $\{C_n\}_{n\in\mathcal{Z}_{T_x}^1}$}
		\EndIf
	\EndLoop
\end{algorithmic}
\end{algorithm}
\begin{rem}
Since $L_x\leq x$, we have 
$\mathbb{P}(R_0\leq z)\geq \mathbb{P}(E/\eta\leq z)=1-\exp(-z\eta)$
for all $z\geq0$.
Since the overshoot $C_{T_x}-x$ is in the interval $(0,d)$,
the expected running time of Algorithm~\ref{alg:SubAlgorithm 5-2} 
(i.e. one over the acceptance probability) 
is smaller than $\exp(\eta (x+d))/(1-\exp(-\eta( x'-x-d))$ if  
$x'>x+d$.
\end{rem}

In Algorithm~\ref{alg:Algorithm 5} we also need to simulate the path of the walk reaching a negative level
$-x$, while staying below a given positive level forever. Algorithm~\ref{alg:SubAlgorithm 5-3} achieves this (see~\citep[Lemma~3]{MR2865624}). 
Its expected running time is bounded above 
by $1/((1-\exp(-\eta(x'+x)))\mathbb{P}(T_{-x}<T_{x^{\prime}}))<\infty$.

\begin{algorithm}
\caption{Simulation of $C_{0},\ldots,C_{T_{-x}}$
conditional on $\left\{ T_{x^{\prime}}=-\infty\right\} $}
\label{alg:SubAlgorithm 5-3} 
\begin{algorithmic}[1]
	\Require{$x \in(0,\infty)$ \& $x^\prime\in(0, \infty]$}
	\Loop
		\State{Sample $C_0=0, C_{-1}, \ldots, C_{T_{-x}}$ under $\mathbb{P}$}
		\State{Given $C_{T_{-x}}$,  sample an independent $1_{\{R_0 ^\prime\leq x^\prime-C_{T_{-x}}\}}$}
			\Comment{Algorithm~\ref{alg:SubAlgorithm 5-1}}
		\If{$1_{\{R_0 ^\prime\leq x^\prime-C_{T_{-x}}\}}=1$ and $\max_{n\in\mathcal{Z}_{T_{-x}}^1}\left\{C_n\right\} \leq x^\prime$}
			\Comment{Accept sample}
			\State{\Return $\{C_n\}_{n\in\mathcal{Z}_{T_{-x}}^1}$}
		\EndIf
	\EndLoop
\end{algorithmic}
\end{algorithm}

We now give a brief overview of the algorithm in~\citep{MR2865624} for the 
simulation of $\left\{ \left(C_{n},R_{n}\right)\right\} _{n\in\mathcal{Z}^{1}}$.
Pick $\kappa>\max\{\log(2)/(3\eta),1/(\alpha\rho)\}$ (see assumption in~\citep[Prop.~3]{MR2865624}).
\citep{MR2865624}~constructs sequences $\Delta=\left\{ \Delta\left(k\right)\right\} _{k\geq0}$
and $\tau=\left\{ \tau_{k}\right\} _{k\geq0}$ of decreasing negative and increasing positive times, respectively: 
\begin{enumerate}
\item at the start of each iteration of the algorithm we are given 
\[
\left(\left\{ \tau_{k}\right\} _{k\in\left\{ 0,\ldots,m\right\} },\left\{
\Delta\left(k\right)\right\} _{k\in\left\{ 0,\ldots,\tau_{m}\right\} },\left\{
C_{n}\right\}_{n\in\mathcal{Z}^1_{\Delta\left(\tau_{m}\right)}},
\left\{ R_{n}\right\} _{n\in\mathcal{Z}^1_{\Delta\left(\tau_{m}-1\right)} }\right), \]
\item at each iteration we sample 
\[
\left(\tau_{m+1},\left\{ \Delta\left(k\right)\right\} _{k\in\left\{ \tau_{m}+1,\ldots,\tau_{m+1}\right\} },
\left\{ C_{n}\right\} _{n\in\mathcal{Z}^{\Delta\left(\tau_{m}\right)}_{\Delta\left(\tau_{m+1}\right)}},
\left\{ R_{n}\right\}_{n\in\mathcal{Z}^{\Delta\left(\tau_{m}-1\right)}_{\Delta\left(\tau_{m+1}-1\right)}}
\right).
\]
\end{enumerate}
Note that at the $m$-th iteration 
we have $\Delta\left(\tau_{m}\right)-\Delta\left(\tau_{m}-1\right)$
more values of the walk than of the reflected process. 
More precisely, the algorithm  starts by setting
$\Delta(0)=0$ and repeats the following steps:
given $\left\{ \tau_{k}\right\} _{k\in\left\{ 0,\ldots,m\right\} }$
and $\left\{ \Delta(k)\right\} _{k\in\left\{ 0,\ldots,\tau_{m}\right\} }$,
then put $\Delta\left(\tau_{m}+1\right)=T_{-2\kappa}^{\Delta\left(\tau_{m}\right)}$.
Next, if $\Delta\left(k\right)$ is the last known value of $\Delta$
and if $R_{\Delta\left(k\right)}>\kappa$, then put $\Delta\left(k+1\right)=T_{\kappa}^{\Delta\left(k\right)}$
and $\Delta\left(k+2\right)=T_{-2\kappa}^{\Delta\left(k+1\right)}$.
If instead $R_{\Delta\left(k\right)}\leq\kappa$ then put $\tau_{m+1}=k$.
Repeat the previous two steps until we can compute $\tau_{m+1}$,
that is, until $R_{\Delta\left(k\right)}\leq\kappa$. After computing
$\tau_{m+1}$ go back and repeat.
By construction (see Proposition 3 in~\citep{MR2865624}) we have
\[
\sup_{n\in\mathcal{Z}^{\Delta\left(\tau_{m}\right)+1}}\left\{ C_{n}\right\} \leq C_{\Delta\left(\tau_{m}-1\right)}-\kappa,\qquad\text{implying}\quad
R_{n}=\max_{k\in\mathcal{Z}_{\Delta\left(\tau_{m}\right)+1}^{n+1}}\left\{ C_{k}\right\} -C_{n},
\quad n\in\mathcal{Z}_{\Delta\left(\tau_{m}-1\right)}^{\Delta\left(\tau_{m-1}\right)}.
\]
Hence, we may compute 
$R_{n}$, $n\in\mathcal{Z}_{\Delta\left(\tau_{m}-1\right)}^{1}$,
from the simulated values $\tau_{m}$, $\Delta_{\left(\tau_{m}-1\right)}$, $\Delta_{\left(\tau_{m}\right)}, 
\left\{ C_{n}\right\} _{n\in\mathcal{Z}_{\Delta\left(\tau_{m}\right)}^{1}}$.
\begin{algorithm}
\caption{Simulation of the random walk and its reflected process}
\label{alg:Algorithm 5} 
\begin{algorithmic}[1]
	\Require{$\kappa > \max\{\frac{\log(2)}{3\eta}, \frac{1}{\alpha\rho}\}$, $d\in (0,1)$, $\infty\geq x>0$ and $m\geq1$}
		\Comment{$x$ is an upper bound for $R_0$}
	\State{Put $t:=C_0:=\Delta(0):=\tau_0:=0$}
	\For{$k\in\{1,\ldots,m\}$}
		\State{Put $t:=\tau_{k-1}$}
		\Loop
			\State{Sample $C_{\Delta(t)-1}, \ldots, C_{T^{\Delta(t)}_{-2\kappa}}$ conditioned on 
				$\{R_{\Delta(t)}<x-C_{\Delta(t)}\}$}
				\Comment{Algorithm~\ref{alg:SubAlgorithm 5-3}}
			\State{Put  $\Delta(t+1):=T^{\Delta(t)}_{-2\kappa}$ and $t:=t+1$}
			\State{Sample $1_{\{R_{\Delta(t)}>\kappa\}}$ given $\{R_{\Delta(t)} < x-C_{\Delta(t)}\}$}
				\Comment{Algorithm~\ref{alg:SubAlgorithm 5-1}}
			\If{$1_{\{R_{\Delta(t)}>\kappa\}}=1$}
				\State{Sample $C_{\Delta(t)-1}, \ldots, C_{T^{\Delta(t)}_\kappa}$ from $\mathbb{P}^\eta$}
					\Comment{Algorithm~\ref{alg:SubAlgorithm 5-2}}
				\State{Put $\Delta(t+1):=T^{\Delta(t)}_\kappa$ and $t:= t+1$}
			\Else
				\State{Put $x:=\kappa+C_{\Delta(t)}$, $\tau_{k}:= t$ and \textbf{exit loop}}
			\EndIf
		\EndLoop
	\EndFor
	\State{Compute $\left\{ R_{n}\right\}_{n\in\mathcal{Z}^1_{\Delta\left(\tau_{m}-1\right)} }$}
	\State{\Return 
$\left(\left\{ \tau_{k}\right\} _{k\in\left\{ 0,\ldots,m\right\} },\left\{
\Delta\left(k\right)\right\} _{k\in\left\{ 0,\ldots,\tau_{m}\right\} },\left\{
C_{n}\right\}_{n\in\mathcal{Z}^1_{\Delta\left(\tau_{m}\right)}},
\left\{ R_{n}\right\} _{n\in\mathcal{Z}^1_{\Delta\left(\tau_{m}-1\right)} }\right)$}
\end{algorithmic}
\end{algorithm}

\subsection{Sampling $(U_n,\Lambda_n)$ given $F_n$}
\label{subsec:ULambda_given_F}
Algorithm~\ref{alg:Algorithm 2}
requires the knowledge of 
$\{(U_n,\Lambda_n)\}_{n\in\mathcal{Z}^0}$,
given the increments $\{F_n\}_{n\in\mathcal{Z}^0}$
of the random walk 
$C$.
Since 
$\log\left(U_{n}\Lambda_{n}\right)=\alpha \left(F_{n}-d\right)$ for all $n\in\mathbb{Z}$,
by independence, we may restrict attention to $n=1$. 
It follows from~\eqref{eq:Distrib_identity} above that 
$\Lambda_{1}\overset{d}{=}\prod_{i=2}^{T}U_{i}$ for an
independent geometric random variable $T$ with parameter $\rho$ on the positive integers (if $T=1$ the right-hand side is defined to equal one). 
Hence, by independence, we have 
$U_1 \Lambda_{1}\overset{d}{=}\prod_{i=1}^{T}U_{i}$.
By~\eqref{eq:log_Lambda_exp},
$-\log \Lambda_1$ conditioned on being positive is exponential with mean $1/\rho$.
Hence for any $n\geq1$ and $y>0$ we obtain
\[
\mathbb{P}\left[ T=n\left|-\sum_{i=1}^{T}\log\left(U_{i}\right)=y\right. \right]
=\frac{\rho\left(1-\rho\right)^{n-1}\frac{y^{n-1}e^{-y}}{\left(n-1\right)!}}{\rho e^{-\rho y}}=\frac{\left[\left(1-\rho\right)y\right]^{n-1}e^{-\left(1-\rho\right)y}}{\left(n-1\right)!}.
\]
Thus the conditional law of $T-1$ given $\sum_{i=1}^{T}\log\left(U_{i}\right)=-y$
is Poisson with mean $\left(1-\rho\right)y$.
If $T=1$, then $-\log(U_1)=y$ and $\Lambda_1=1$. If $T>1$, then for $x\in(0,y)$ we get
\[
\mathbb{P}\left[ -\log\left(U_{1}\right)\in dx\left|T=n,-\sum_{i=1}^{T}\log\left(U_{i}\right)=y\right.\right]
=\frac{e^{-x}\frac{\left(y-x\right)^{n-2}e^{-\left(y-x\right)}}{\left(n-2\right)!}}{\frac{y^{n-1}e^{-y}}{\left(n-1\right)!}}dx
=\left(n-1\right)\frac{\left(y-x\right)^{n-2}}{y^{n-1}}dx.
\]
Hence, conditional on 
$T=n$ and $\log\left(\prod_{i=1}^{T}U_{i}\right)=-y$,
the law of $-\frac{1}{y}\log\left(U_{1}\right)$ 
is $Beta\left(1,n-1\right)$
(understood as the Dirac measure $\delta_{1}$ when $n=1$). 
Finally we set  $\Lambda_{1}=\exp\left(\alpha\left(F_1-d\right)\right)/U_1$. 
\begin{algorithm}
\caption{Simulation of $\left\{\left(U_{k},\Lambda_{k}\right)\right\}_{k\in\mathcal{Z}_{m}^{n}}$
given $\left\{F_{k}\right\}_{k\in\mathcal{Z}_{m}^{n}}$}
\label{alg:Algorithm 2-1}
\begin{algorithmic}[1]
	\Require{$\left\{F_{k}\right\}_{k\in\mathcal{Z}_{m}^{n}}$ for $m,n\in\mathbb{Z}$ and $m< n$.}
	\For{$k\in \mathcal{Z}^n_m$}
		\State{Sample $T-1\sim Poisson\left(-\alpha\left(F_{k}-d\right)\left(1-\rho\right)\right)$}
		\State{Sample $L\sim Beta\left(1,T-1\right)$ }
		\State{Let $U_k:=\exp\left(L\alpha\left(F_k-d\right)\right)$ and $\Lambda_k := \exp\left(\left(1-L\right)\alpha\left(F_{k}-d\right)\right)$}
	\EndFor
	\State{\Return $\left\{\left(U_{k},\Lambda_{k}\right)\right\}_{k\in\mathcal{Z}_{m}^{n}}$}
\end{algorithmic}
\end{algorithm}
\section{\label{sec:numerics}Implementation}

Recall the definitions of the process $\{(C_n,F_n)\}_{n\in\mathbb{Z}}$ in~\ref{eq:RW}, 
of $\{\Theta_n\}_{n\in\mathbb{Z}}$ in the first paragraph of Section~\ref{sec:Backward_Sim} 
and of $\mathbb{P}^\eta$ in the second paragraph of Section~\ref{subsec:BlanchetSigman}. 
Before providing a concrete and concise algorithm and testing it, 
we will introduce a practical improvement 
based on a simple consideration.

Note that simulating the iid variables $\{\Theta_n\}_{n\in\mathcal{Z}^0}$ is 
clearly quicker and easier than employing the full machinery of our algorithms.
Recall that the dominating process was introduced only to detect 
coalescence for the chain $\{X_n\}_{n\in\mathcal{Z}^0}$. Thus, given 
$\{\Theta_n\}_{n\in\mathcal{Z}_{\Delta(0)}^0}$ for some burn-in parameter 
$\Delta(0)\in\mathcal{Z}^0$ and an upper bound 
$X^\prime_{\Delta(0)}=D_{\Delta(0)}\geq X_{\Delta(0)}$
(recall the definition of $\{D_n\}$ in~\eqref{eq:DominatingProcess}), 
one could recursively construct $X_{n+1}^\prime=\psi(X_n^\prime,\Theta_n)$ 
for $n\in\mathcal{Z}_{\Delta(0)}^{0}$ and if any coalescence were detected, 
we would be certain that $X_0^\prime=X_0$. Our objective is hence to 
take an appropriate $\Delta(0)$ that increases the probability 
$\mathbb{P}(X_0^\prime=X_0)$. Algorithm~\ref{alg:FullAlgorithm} is a complete 
and compact simulation algorithm of $X_0$, which makes use of this.

It is known that spectrally negative stable processes of infinite variation
($\alpha>1$ and $\rho=1/\alpha$) satisfy 
$\overline{S}(\alpha,\rho)=S^+(\alpha,\rho)$~\citep[Thm~1]{MR3033593}.
As a simple application and sanity-check, we now present a comparison
between the empirical distribution function of $N=10^4$ samples 
against the actual distribution function in this case.
To validate the samples, we compute the Kolmogorov-Smirnov statistic
and test the hypothesis.\footnote{These graphs can be replicated 
following the guide available in~\cite{Jorge_GitHub}.} In all three cases 
the null hypothesis of all samples coming from their respective
distribution functions is not rejected (see Figure~\ref{fig:ecdf}).

\begin{figure}[H]
		\centering{}\resizebox{.32\linewidth}{!}{
		\begin{subfigure}[h]{.75\linewidth}
			\begin{tikzpicture} 
			\begin{axis} 
		[
		ymin=0,
		ymax=1,
		xmin=0,
		xmax=1.6,
		title={\LARGE$\alpha=1.1$},
		xlabel={\LARGE $x$},
		ylabel={\LARGE $F(x)$},
		width=12cm,
		height=7cm,
		axis on top=true,
		axis x line=middle,
		axis y line=middle,
		legend style = {at={(1,.2)}, anchor=south east}
		]
		
		\addplot[
		dashed,
		mark=o,
		color=black,
		]
		coordinates {
		(0.0,0.0)(0.05,0.0049)(0.1,0.0103)(0.15,0.0163)(0.2,0.0228)(0.25,0.0300)(0.3,0.0379)(0.35,0.0468)(0.4,0.0566)(0.45,0.0676)(0.5,0.0800)(0.55,0.0940)(0.6,0.1098)(0.65,0.1279)(0.7,0.1486)(0.75,0.1724)(0.8,0.1998)(0.85,0.2317)(0.9,0.2686)(0.95,0.3116)(1.0,0.3610)(1.05,0.4183)(1.1,0.4835)(1.15,0.5565)(1.2,0.6360)(1.25,0.7188)(1.3,0.8000)(1.35,0.8728)(1.4,0.9306)(1.45,0.9692)(1.5,0.9896)(1.55,0.9976)(1.6,0.9996)
		};
		
		\addplot[
		dashed,
		mark=+,
		color=black,
		]
		coordinates {
		(0.0,0.0)(0.05,0.005)(0.1,0.009)(0.15,0.016)(0.2,0.022)(0.25,0.029)(0.3,0.037)(0.35,0.046)(0.4,0.056)(0.45,0.068)(0.5,0.080)(0.55,0.093)(0.6,0.109)(0.65,0.126)(0.7,0.147)(0.75,0.168)(0.8,0.196)(0.85,0.231)(0.9,0.264)(0.95,0.308)(1.0,0.359)(1.05,0.413)(1.1,0.482)(1.15,0.549)(1.2,0.628)(1.25,0.712)(1.3,0.794)(1.35,0.87)(1.4,0.930)(1.45,0.970)(1.5,0.990)(1.55,0.997)(1.6,0.999)
		};
		
		\legend {\LARGE Real,\LARGE Estimate};
		\end{axis}
			\end{tikzpicture}
		\end{subfigure}}
		\centering{}\resizebox{.32\linewidth}{!}{
		\begin{subfigure}[h]{.75\linewidth}
			\begin{tikzpicture} 
			\begin{axis} 
		[
		ymin=0,
		ymax=1,
		xmin=0,
		xmax=2.6,
		title={\LARGE$\alpha=1.5$},
		xlabel={\LARGE $x$},
		ylabel={\LARGE $F(x)$},
		width=12cm,
		height=7cm,
		axis on top=true,
		axis x line=middle,
		axis y line=middle,
		legend style = {at={(1,.2)}, anchor=south east}
		]
		
		\addplot[
		dashed,
		mark=o,
		color=black,
		]
		coordinates {
		(0.0,0.0)(0.05,0.0189)(0.1,0.0385)(0.15,0.0587)(0.2,0.0795)(0.25,0.1009)(0.3,0.1229)(0.35,0.1454)(0.4,0.1684)(0.45,0.1920)(0.5,0.2161)(0.55,0.2406)(0.6,0.2655)(0.65,0.2908)(0.7,0.3164)(0.75,0.3423)(0.8,0.3684)(0.85,0.3946)(0.9,0.4210)(0.95,0.4474)(1.0,0.4737)(1.05,0.4999)(1.1,0.5260)(1.15,0.5518)(1.2,0.5772)(1.25,0.6023)(1.3,0.6269)(1.35,0.6510)(1.4,0.6745)(1.45,0.6973)(1.5,0.7193)(1.55,0.7407)(1.6,0.7611)(1.65,0.7807)(1.7,0.7994)(1.75,0.8172)(1.8,0.8340)(1.85,0.8499)(1.9,0.8647)(1.95,0.8786)(2.0,0.8915)(2.05,0.9034)(2.1,0.9144)(2.15,0.9245)(2.2,0.9337)(2.25,0.9420)(2.3,0.9496)(2.35,0.9563)(2.4,0.9624)(2.45,0.9678)(2.5,0.9725)(2.55,0.9767)(2.6,0.9803)
		};
		
		\addplot[
		dashed,
		mark=+,
		color=black,
		]
		coordinates {
		(0.0,0.0)(0.05,0.020)(0.1,0.038)(0.15,0.059)(0.2,0.081)(0.25,0.101)(0.3,0.124)(0.35,0.148)(0.4,0.172)(0.45,0.196)(0.5,0.220)(0.55,0.243)(0.6,0.265)(0.65,0.290)(0.7,0.316)(0.75,0.343)(0.8,0.370)(0.85,0.398)(0.9,0.424)(0.95,0.449)(1.0,0.474)(1.05,0.503)(1.1,0.529)(1.15,0.555)(1.2,0.581)(1.25,0.606)(1.3,0.633)(1.35,0.657)(1.4,0.679)(1.45,0.701)(1.5,0.724)(1.55,0.746)(1.6,0.768)(1.65,0.789)(1.7,0.808)(1.75,0.826)(1.8,0.841)(1.85,0.856)(1.9,0.871)(1.95,0.886)(2.0,0.899)(2.05,0.910)(2.1,0.919)(2.15,0.930)(2.2,0.938)(2.25,0.946)(2.3,0.954)(2.35,0.959)(2.4,0.964)(2.45,0.968)(2.5,0.974)(2.55,0.977)(2.6,0.981)
		};
		
		\legend {\LARGE Real,\LARGE Estimate};
		\end{axis}
			\end{tikzpicture}
		\end{subfigure}}
		\centering{}\resizebox{.32\linewidth}{!}{
		\begin{subfigure}[h]{.75\linewidth}
			\begin{tikzpicture} 
			\begin{axis} 
		[
		ymin=0,
		ymax=1,
		xmin=0,
		xmax=3.6,
		title={\LARGE$\alpha=1.9$},
		xlabel={\LARGE $x$},
		ylabel={\LARGE $F(x)$},
		width=12cm,
		height=7cm,
		axis on top=true,
		axis x line=middle,
		axis y line=middle,
		legend style = {at={(1,.2)}, anchor=south east}
		]
		
		\addplot[
		dashed,
		mark=o,
		color=black,
		]
		coordinates {
		(0.0,0.0)(0.05,0.0267)(0.1,0.0536)(0.15,0.0806)(0.2,0.1076)(0.25,0.1345)(0.3,0.1614)(0.35,0.1882)(0.4,0.2149)(0.45,0.2415)(0.5,0.2678)(0.55,0.2940)(0.6,0.3198)(0.65,0.3454)(0.7,0.3706)(0.75,0.3955)(0.8,0.4200)(0.85,0.4441)(0.9,0.4678)(0.95,0.4910)(1.0,0.5137)(1.05,0.5359)(1.1,0.5576)(1.15,0.5787)(1.2,0.5993)(1.25,0.6193)(1.3,0.6387)(1.35,0.6575)(1.4,0.6757)(1.45,0.6933)(1.5,0.7103)(1.55,0.7267)(1.6,0.7425)(1.65,0.7576)(1.7,0.7721)(1.75,0.7860)(1.8,0.7994)(1.85,0.8121)(1.9,0.8242)(1.95,0.8358)(2.0,0.8467)(2.05,0.8572)(2.1,0.8671)(2.15,0.8764)(2.2,0.8853)(2.25,0.8937)(2.3,0.9015)(2.35,0.9089)(2.4,0.9159)(2.45,0.9225)(2.5,0.9286)(2.55,0.9343)(2.6,0.9396)(2.65,0.9446)(2.7,0.9493)(2.75,0.9536)(2.8,0.9576)(2.85,0.9613)(2.9,0.9648)(2.95,0.9679)(3.0,0.9709)(3.05,0.9736)(3.1,0.9760)(3.15,0.9783)(3.2,0.9804)(3.25,0.9823)(3.3,0.9841)(3.35,0.9857)(3.4,0.9871)(3.45,0.9884)(3.5,0.9896)(3.55,0.9907)(3.6,0.9917)
		};
		
		\addplot[
		dashed,
		mark=+,
		color=black,
		]
		coordinates {
		(0.0,0.0)(0.05,0.025)(0.1,0.049)(0.15,0.075)(0.2,0.100)(0.25,0.129)(0.3,0.157)(0.35,0.183)(0.4,0.213)(0.45,0.238)(0.5,0.264)(0.55,0.287)(0.6,0.317)(0.65,0.342)(0.7,0.368)(0.75,0.395)(0.8,0.418)(0.85,0.447)(0.9,0.471)(0.95,0.493)(1.0,0.516)(1.05,0.536)(1.1,0.557)(1.15,0.578)(1.2,0.595)(1.25,0.613)(1.3,0.634)(1.35,0.652)(1.4,0.671)(1.45,0.688)(1.5,0.705)(1.55,0.721)(1.6,0.739)(1.65,0.753)(1.7,0.767)(1.75,0.781)(1.8,0.795)(1.85,0.807)(1.9,0.819)(1.95,0.831)(2.0,0.840)(2.05,0.851)(2.1,0.862)(2.15,0.872)(2.2,0.880)(2.25,0.891)(2.3,0.900)(2.35,0.907)(2.4,0.915)(2.45,0.921)(2.5,0.928)(2.55,0.932)(2.6,0.938)(2.65,0.943)(2.7,0.948)(2.75,0.953)(2.8,0.956)(2.85,0.961)(2.9,0.964)(2.95,0.967)(3.0,0.971)(3.05,0.973)(3.1,0.975)(3.15,0.977)(3.2,0.979)(3.25,0.981)(3.3,0.983)(3.35,0.985)(3.4,0.986)(3.45,0.987)(3.5,0.988)(3.55,0.99)(3.6,0.991)
		};
		
		\legend {\LARGE Real,\LARGE Estimate};
		\end{axis}
			\end{tikzpicture}
		\end{subfigure}}
		\centering{}\resizebox{.32\linewidth}{!}{
		\begin{subfigure}[b]{.75\linewidth}
			\begin{tikzpicture} 
			\begin{axis} 
		[
		ymin=-1.7,
		ymax=1.7,
		xmin=0,
		xmax=1,
		xlabel={\LARGE $F(t)$},
		ylabel={\LARGE $\sqrt{n}(F_n(t)-F(t))$},
		width=12cm,
		height=7cm,
		axis on top=true,
		axis x line=middle,
		axis y line=middle,	
		legend style = {at={(1,.8)}, anchor=south east}
		]
		
		\addplot[
		dashed,		
		]
		coordinates{
		(0,1.2221)(1,1.2221)
		};
		\addplot[
		dashed,		
		]
		coordinates{
		(0,-1.2221)(1,-1.2221)
		};
		
		\addplot[
		solid,
		color=black,
		]
		coordinates {
		(0,0)(0.000,0.0165)(0.002,0.0290)(0.004,0.0784)(0.006,0.0414)(0.009,-0.016)(0.012,-0.105)(0.014,-0.113)(0.015,-0.035)(0.017,-0.052)(0.018,0.0185)(0.020,0.0519)(0.023,-0.016)(0.025,-0.057)(0.027,-0.066)(0.029,-0.068)(0.031,-0.038)(0.033,-0.018)(0.036,-0.117)(0.037,-0.092)(0.040,-0.150)(0.042,-0.102)(0.043,-0.062)(0.044,0.0145)(0.046,0.0053)(0.048,0.0611)(0.050,0.0885)(0.052,0.0482)(0.054,0.0310)(0.056,0.0094)(0.058,0.0651)(0.060,0.0188)(0.063,-0.008)(0.064,0.0307)(0.066,0.0623)(0.068,0.0500)(0.070,0.0202)(0.073,-0.071)(0.075,-0.009)(0.077,-0.016)(0.079,-0.004)(0.080,0.0196)(0.082,0.0465)(0.084,0.0575)(0.086,0.0076)(0.089,-0.043)(0.091,-0.066)(0.093,-0.055)(0.095,-0.029)(0.097,-0.053)(0.100,-0.114)(0.101,-0.057)(0.103,-0.078)(0.105,-0.059)(0.107,-0.099)(0.109,-0.054)(0.112,-0.110)(0.114,-0.117)(0.116,-0.122)(0.118,-0.157)(0.120,-0.154)(0.122,-0.143)(0.124,-0.146)(0.126,-0.176)(0.128,-0.184)(0.130,-0.199)(0.132,-0.164)(0.134,-0.147)(0.136,-0.133)(0.138,-0.176)(0.140,-0.150)(0.142,-0.171)(0.145,-0.208)(0.146,-0.129)(0.147,-0.085)(0.149,-0.081)(0.152,-0.153)(0.154,-0.158)(0.156,-0.185)(0.158,-0.191)(0.161,-0.243)(0.163,-0.237)(0.165,-0.211)(0.168,-0.367)(0.170,-0.360)(0.172,-0.363)(0.174,-0.325)(0.176,-0.366)(0.178,-0.315)(0.180,-0.339)(0.182,-0.385)(0.184,-0.398)(0.187,-0.424)(0.189,-0.401)(0.190,-0.393)(0.193,-0.427)(0.194,-0.367)(0.196,-0.381)(0.198,-0.391)(0.200,-0.366)(0.202,-0.330)(0.204,-0.387)(0.206,-0.385)(0.208,-0.337)(0.209,-0.290)(0.211,-0.279)(0.214,-0.309)(0.215,-0.299)(0.217,-0.285)(0.219,-0.245)(0.221,-0.203)(0.222,-0.192)(0.224,-0.149)(0.225,-0.098)(0.228,-0.110)(0.229,-0.078)(0.231,-0.041)(0.233,-0.057)(0.236,-0.146)(0.238,-0.173)(0.241,-0.271)(0.243,-0.272)(0.245,-0.260)(0.247,-0.238)(0.249,-0.252)(0.252,-0.321)(0.254,-0.323)(0.256,-0.302)(0.258,-0.319)(0.260,-0.358)(0.263,-0.428)(0.264,-0.375)(0.267,-0.441)(0.269,-0.404)(0.271,-0.405)(0.272,-0.396)(0.274,-0.380)(0.278,-0.507)(0.280,-0.579)(0.283,-0.671)(0.284,-0.593)(0.286,-0.527)(0.288,-0.502)(0.289,-0.492)(0.291,-0.433)(0.293,-0.439)(0.295,-0.420)(0.296,-0.363)(0.298,-0.351)(0.300,-0.306)(0.301,-0.286)(0.303,-0.244)(0.305,-0.288)(0.307,-0.296)(0.309,-0.284)(0.312,-0.328)(0.314,-0.349)(0.315,-0.271)(0.317,-0.227)(0.319,-0.279)(0.321,-0.240)(0.323,-0.284)(0.325,-0.224)(0.326,-0.154)(0.328,-0.139)(0.330,-0.135)(0.332,-0.123)(0.334,-0.131)(0.336,-0.129)(0.338,-0.197)(0.340,-0.129)(0.342,-0.158)(0.344,-0.162)(0.345,-0.085)(0.348,-0.101)(0.349,-0.083)(0.352,-0.101)(0.354,-0.100)(0.355,-0.079)(0.358,-0.122)(0.360,-0.175)(0.362,-0.160)(0.364,-0.145)(0.365,-0.099)(0.367,-0.069)(0.370,-0.142)(0.372,-0.111)(0.374,-0.157)(0.376,-0.174)(0.378,-0.165)(0.380,-0.197)(0.383,-0.202)(0.385,-0.282)(0.387,-0.296)(0.390,-0.339)(0.393,-0.462)(0.395,-0.488)(0.397,-0.451)(0.399,-0.467)(0.401,-0.497)(0.404,-0.518)(0.405,-0.483)(0.407,-0.454)(0.410,-0.550)(0.412,-0.507)(0.413,-0.481)(0.415,-0.448)(0.417,-0.464)(0.419,-0.481)(0.422,-0.507)(0.424,-0.509)(0.426,-0.528)(0.427,-0.484)(0.429,-0.468)(0.431,-0.466)(0.434,-0.532)(0.435,-0.482)(0.437,-0.429)(0.439,-0.431)(0.441,-0.459)(0.444,-0.517)(0.445,-0.484)(0.447,-0.497)(0.449,-0.488)(0.451,-0.458)(0.453,-0.497)(0.455,-0.406)(0.457,-0.478)(0.459,-0.411)(0.460,-0.397)(0.462,-0.367)(0.464,-0.319)(0.465,-0.286)(0.467,-0.213)(0.469,-0.231)(0.471,-0.240)(0.473,-0.211)(0.475,-0.237)(0.477,-0.201)(0.478,-0.185)(0.480,-0.145)(0.482,-0.121)(0.485,-0.236)(0.486,-0.187)(0.489,-0.247)(0.491,-0.285)(0.493,-0.270)(0.495,-0.270)(0.498,-0.312)(0.500,-0.384)(0.503,-0.443)(0.505,-0.438)(0.508,-0.579)(0.510,-0.565)(0.512,-0.502)(0.514,-0.529)(0.516,-0.537)(0.518,-0.525)(0.519,-0.499)(0.521,-0.455)(0.524,-0.561)(0.526,-0.580)(0.529,-0.632)(0.531,-0.633)(0.533,-0.663)(0.535,-0.611)(0.536,-0.590)(0.539,-0.625)(0.541,-0.651)(0.543,-0.655)(0.545,-0.668)(0.547,-0.678)(0.549,-0.682)(0.551,-0.686)(0.554,-0.732)(0.555,-0.697)(0.557,-0.694)(0.559,-0.626)(0.560,-0.550)(0.562,-0.549)(0.564,-0.528)(0.566,-0.539)(0.568,-0.521)(0.569,-0.488)(0.572,-0.519)(0.574,-0.505)(0.576,-0.544)(0.578,-0.560)(0.580,-0.564)(0.583,-0.651)(0.585,-0.630)(0.587,-0.627)(0.588,-0.570)(0.590,-0.559)(0.592,-0.561)(0.594,-0.560)(0.596,-0.506)(0.598,-0.537)(0.600,-0.559)(0.602,-0.548)(0.605,-0.600)(0.606,-0.588)(0.608,-0.544)(0.610,-0.569)(0.612,-0.572)(0.614,-0.574)(0.616,-0.562)(0.618,-0.515)(0.620,-0.567)(0.622,-0.595)(0.624,-0.545)(0.627,-0.624)(0.629,-0.658)(0.631,-0.675)(0.634,-0.763)(0.636,-0.739)(0.639,-0.845)(0.641,-0.840)(0.642,-0.797)(0.644,-0.741)(0.646,-0.706)(0.648,-0.742)(0.649,-0.678)(0.651,-0.616)(0.652,-0.595)(0.655,-0.602)(0.657,-0.642)(0.659,-0.646)(0.661,-0.624)(0.662,-0.581)(0.665,-0.620)(0.666,-0.525)(0.668,-0.552)(0.670,-0.551)(0.672,-0.548)(0.673,-0.473)(0.676,-0.514)(0.678,-0.538)(0.680,-0.534)(0.682,-0.521)(0.684,-0.534)(0.686,-0.532)(0.688,-0.503)(0.690,-0.515)(0.692,-0.560)(0.694,-0.554)(0.695,-0.468)(0.697,-0.481)(0.699,-0.483)(0.701,-0.465)(0.704,-0.577)(0.707,-0.695)(0.710,-0.705)(0.712,-0.759)(0.714,-0.724)(0.716,-0.707)(0.717,-0.680)(0.720,-0.710)(0.722,-0.700)(0.723,-0.676)(0.725,-0.634)(0.726,-0.564)(0.728,-0.586)(0.731,-0.611)(0.732,-0.582)(0.734,-0.588)(0.737,-0.653)(0.739,-0.606)(0.741,-0.663)(0.743,-0.642)(0.744,-0.591)(0.746,-0.572)(0.748,-0.562)(0.750,-0.578)(0.752,-0.574)(0.755,-0.653)(0.756,-0.595)(0.759,-0.638)(0.761,-0.641)(0.763,-0.616)(0.764,-0.595)(0.766,-0.518)(0.768,-0.554)(0.771,-0.617)(0.772,-0.599)(0.775,-0.622)(0.777,-0.651)(0.778,-0.586)(0.780,-0.590)(0.782,-0.559)(0.785,-0.631)(0.787,-0.609)(0.789,-0.614)(0.790,-0.572)(0.793,-0.618)(0.794,-0.576)(0.795,-0.488)(0.797,-0.429)(0.800,-0.506)(0.801,-0.480)(0.803,-0.446)(0.805,-0.461)(0.807,-0.442)(0.808,-0.387)(0.811,-0.425)(0.813,-0.422)(0.814,-0.354)(0.817,-0.400)(0.819,-0.450)(0.821,-0.468)(0.823,-0.445)(0.824,-0.385)(0.826,-0.337)(0.829,-0.433)(0.830,-0.379)(0.832,-0.316)(0.833,-0.278)(0.836,-0.350)(0.838,-0.319)(0.840,-0.328)(0.842,-0.364)(0.843,-0.281)(0.845,-0.255)(0.847,-0.216)(0.848,-0.194)(0.850,-0.186)(0.853,-0.243)(0.855,-0.249)(0.857,-0.205)(0.859,-0.221)(0.861,-0.296)(0.863,-0.224)(0.865,-0.220)(0.867,-0.238)(0.869,-0.241)(0.871,-0.252)(0.873,-0.267)(0.876,-0.317)(0.877,-0.279)(0.880,-0.369)(0.883,-0.409)(0.884,-0.341)(0.886,-0.381)(0.887,-0.272)(0.889,-0.296)(0.891,-0.258)(0.893,-0.225)(0.895,-0.234)(0.896,-0.185)(0.899,-0.205)(0.900,-0.198)(0.902,-0.151)(0.904,-0.144)(0.906,-0.116)(0.907,-0.073)(0.909,-0.060)(0.911,-0.026)(0.913,-0.004)(0.914,0.0230)(0.916,0.0346)(0.918,0.0079)(0.920,0.0330)(0.923,-0.055)(0.925,-0.034)(0.927,-0.036)(0.929,-0.056)(0.930,0.0271)(0.932,0.0513)(0.934,0.0154)(0.936,0.0100)(0.938,0.0824)(0.940,0.0597)(0.942,0.0828)(0.944,0.0753)(0.947,-0.054)(0.949,-0.009)(0.950,0.0568)(0.952,0.0822)(0.954,0.0848)(0.956,0.0970)(0.958,0.0538)(0.960,0.0836)(0.962,0.0733)(0.964,0.0413)(0.966,0.0851)(0.967,0.1208)(0.969,0.1073)(0.972,0.0476)(0.975,-0.0043)(0.977,-0.000)(0.978,0.0228)(0.981,-0.002)(0.982,0.0025)(0.984,0.0547)(0.986,0.0945)(0.987,0.1099)(0.990,0.0896)(0.992,0.0688)(0.995,-0.011)(0.997,-0.021)(0.999,-0.024)(1,0)
		};
		
		
		\end{axis}
			\end{tikzpicture}
		\end{subfigure}}
		\centering{}\resizebox{.32\linewidth}{!}{
		\begin{subfigure}[b]{.75\linewidth}
			\begin{tikzpicture} 
			\begin{axis} 
		[
		ymin=-1.7,
		ymax=1.7,
		xmin=0,
		xmax=1,
		xlabel={\LARGE $F(t)$},
		ylabel={\LARGE $\sqrt{n}(F_n(t)-F(t))$},
		width=12.5cm,
		height=7cm,
		axis on top=true,
		axis x line=middle,
		axis y line=middle,	
		legend style = {at={(1,.8)}, anchor=south east}
		]
		
		\addplot[
		dashed,		
		]
		coordinates{
		(0,1.2221)(1,1.2221)
		};
		\addplot[
		dashed,		
		]
		coordinates{
		(0,-1.2221)(1,-1.2221)
		};
		
		\addplot[
		solid,
		color=black,
		]
		coordinates {
		(0,0)(0.001,-0.011)(0.002,0.0253)(0.004,0.0053)(0.006,0.0638)(0.008,0.0367)(0.010,0.0908)(0.011,0.1080)(0.013,0.1062)(0.016,0.0702)(0.018,0.0916)(0.019,0.1524)(0.020,0.2140)(0.022,0.2228)(0.024,0.2337)(0.027,0.1662)(0.030,0.0631)(0.032,0.0974)(0.033,0.1013)(0.036,0.0748)(0.038,0.0404)(0.040,0.0032)(0.042,0.0388)(0.044,0.0819)(0.046,0.0049)(0.048,0.0324)(0.050,0.0588)(0.052,0.0698)(0.054,0.0834)(0.056,0.0823)(0.057,0.1463)(0.059,0.1616)(0.060,0.2362)(0.063,0.1849)(0.065,0.1942)(0.067,0.1565)(0.069,0.1673)(0.071,0.1709)(0.073,0.1880)(0.074,0.2117)(0.076,0.2077)(0.078,0.2385)(0.080,0.2533)(0.082,0.2200)(0.085,0.1905)(0.087,0.1116)(0.090,0.0803)(0.092,0.0587)(0.094,0.0366)(0.096,0.0585)(0.098,0.0941)(0.100,0.0934)(0.101,0.1048)(0.104,0.0863)(0.106,0.0966)(0.108,0.0606)(0.109,0.1036)(0.112,0.0840)(0.114,0.0850)(0.116,0.0891)(0.117,0.1146)(0.120,0.0694)(0.121,0.1141)(0.123,0.1676)(0.125,0.1485)(0.127,0.1948)(0.128,0.2940)(0.129,0.3514)(0.131,0.3553)(0.133,0.3618)(0.135,0.3001)(0.138,0.2713)(0.140,0.2481)(0.142,0.2640)(0.144,0.2623)(0.146,0.2356)(0.148,0.2732)(0.150,0.2784)(0.151,0.3146)(0.153,0.3564)(0.155,0.3320)(0.157,0.3231)(0.160,0.2694)(0.162,0.2676)(0.164,0.2834)(0.165,0.3624)(0.167,0.3567)(0.168,0.4251)(0.171,0.3613)(0.173,0.3383)(0.176,0.2597)(0.178,0.2268)(0.180,0.2886)(0.181,0.3139)(0.183,0.3070)(0.185,0.3149)(0.187,0.3341)(0.189,0.3824)(0.191,0.3972)(0.192,0.4388)(0.194,0.4803)(0.196,0.4590)(0.198,0.4971)(0.199,0.5298)(0.201,0.5025)(0.203,0.5090)(0.205,0.5139)(0.208,0.4614)(0.209,0.5155)(0.211,0.5442)(0.213,0.5294)(0.216,0.4126)(0.219,0.3878)(0.221,0.3735)(0.223,0.3251)(0.225,0.3101)(0.227,0.3243)(0.230,0.2912)(0.231,0.3068)(0.234,0.2770)(0.236,0.2927)(0.238,0.2817)(0.240,0.2352)(0.243,0.1169)(0.246,0.0439)(0.248,0.0337)(0.250,0.0013)(0.252,0.0284)(0.254,0.0125)(0.256,0.0241)(0.259,-0.035)(0.261,-0.011)(0.262,0.0341)(0.265,-0.046)(0.268,-0.173)(0.271,-0.211)(0.272,-0.140)(0.274,-0.146)(0.275,-0.083)(0.278,-0.109)(0.279,-0.094)(0.281,-0.020)(0.283,-0.068)(0.285,-0.057)(0.287,-0.073)(0.289,-0.085)(0.292,-0.118)(0.293,-0.069)(0.295,-0.062)(0.296,0.0316)(0.299,-0.002)(0.300,0.0397)(0.302,0.0702)(0.304,0.0449)(0.306,0.0413)(0.308,0.0034)(0.311,-0.044)(0.313,-0.035)(0.315,-0.048)(0.317,-0.027)(0.319,-0.040)(0.321,-0.015)(0.322,0.0012)(0.324,0.0414)(0.326,0.0704)(0.328,0.0251)(0.330,0.0244)(0.333,-0.005)(0.334,0.0648)(0.336,0.0026)(0.338,0.0471)(0.340,0.0564)(0.341,0.1204)(0.344,0.0894)(0.346,0.0731)(0.347,0.1138)(0.349,0.1222)(0.351,0.1524)(0.353,0.1785)(0.355,0.1963)(0.358,0.0884)(0.359,0.1601)(0.361,0.1838)(0.363,0.1870)(0.365,0.1954)(0.366,0.2660)(0.368,0.2367)(0.370,0.2586)(0.371,0.3386)(0.373,0.3499)(0.376,0.2455)(0.378,0.2125)(0.380,0.2680)(0.382,0.2769)(0.384,0.2730)(0.386,0.2814)(0.388,0.2536)(0.390,0.2949)(0.391,0.3822)(0.393,0.3564)(0.395,0.3779)(0.397,0.3940)(0.399,0.3626)(0.400,0.4660)(0.403,0.3357)(0.405,0.3197)(0.407,0.3596)(0.409,0.3265)(0.411,0.3764)(0.413,0.3458)(0.415,0.3850)(0.416,0.4303)(0.419,0.3527)(0.421,0.3765)(0.423,0.3959)(0.425,0.3875)(0.426,0.4318)(0.428,0.4618)(0.431,0.3612)(0.433,0.3866)(0.435,0.3573)(0.438,0.2751)(0.440,0.2502)(0.442,0.2675)(0.444,0.2341)(0.447,0.1912)(0.449,0.1580)(0.451,0.1662)(0.453,0.1005)(0.455,0.1051)(0.457,0.1276)(0.460,0.0746)(0.462,0.0515)(0.464,0.0587)(0.466,0.0591)(0.467,0.1133)(0.470,0.0979)(0.471,0.1317)(0.473,0.1250)(0.475,0.1291)(0.477,0.1614)(0.479,0.1081)(0.481,0.1617)(0.483,0.1646)(0.484,0.2213)(0.486,0.2840)(0.487,0.3679)(0.488,0.4409)(0.490,0.4754)(0.493,0.3721)(0.494,0.4062)(0.496,0.4727)(0.499,0.3888)(0.501,0.3783)(0.504,0.2878)(0.505,0.3140)(0.507,0.3716)(0.509,0.3185)(0.511,0.3254)(0.514,0.2893)(0.515,0.3648)(0.517,0.3700)(0.519,0.3979)(0.521,0.3782)(0.523,0.3201)(0.525,0.3551)(0.528,0.2839)(0.529,0.3111)(0.531,0.3048)(0.533,0.3632)(0.535,0.3954)(0.536,0.4270)(0.538,0.4536)(0.540,0.4442)(0.542,0.4099)(0.544,0.4223)(0.546,0.4127)(0.548,0.4221)(0.551,0.3845)(0.553,0.3444)(0.555,0.3343)(0.557,0.3746)(0.559,0.3695)(0.561,0.3843)(0.562,0.4024)(0.564,0.4741)(0.567,0.3732)(0.569,0.3690)(0.570,0.4276)(0.572,0.4734)(0.574,0.4550)(0.575,0.5000)(0.578,0.4723)(0.580,0.4282)(0.582,0.4559)(0.585,0.3808)(0.588,0.2477)(0.589,0.3162)(0.591,0.3967)(0.592,0.4321)(0.594,0.4241)(0.596,0.4042)(0.598,0.4119)(0.600,0.4562)(0.602,0.4033)(0.604,0.4400)(0.605,0.5445)(0.607,0.5814)(0.608,0.6034)(0.610,0.6619)(0.612,0.6472)(0.615,0.5957)(0.617,0.5941)(0.618,0.6246)(0.620,0.6730)(0.622,0.6069)(0.624,0.6918)(0.626,0.6817)(0.628,0.6799)(0.630,0.6992)(0.632,0.6860)(0.633,0.7851)(0.635,0.7576)(0.637,0.7651)(0.638,0.8113)(0.640,0.8113)(0.643,0.7109)(0.645,0.7102)(0.648,0.6836)(0.650,0.6958)(0.651,0.7008)(0.654,0.6011)(0.657,0.5785)(0.659,0.5195)(0.662,0.4787)(0.664,0.4494)(0.666,0.4878)(0.668,0.4850)(0.670,0.4708)(0.671,0.5143)(0.674,0.4785)(0.676,0.4570)(0.677,0.5413)(0.679,0.5703)(0.681,0.5549)(0.683,0.5592)(0.686,0.4972)(0.688,0.4958)(0.691,0.3954)(0.693,0.3444)(0.695,0.3758)(0.696,0.4247)(0.698,0.4871)(0.700,0.4405)(0.702,0.4467)(0.704,0.4294)(0.707,0.3885)(0.708,0.4739)(0.709,0.5432)(0.710,0.6001)(0.713,0.5421)(0.715,0.5133)(0.717,0.5436)(0.719,0.5568)(0.722,0.4830)(0.724,0.4250)(0.725,0.5035)(0.727,0.5595)(0.729,0.5634)(0.732,0.4397)(0.734,0.4948)(0.736,0.4939)(0.736,0.6001)(0.739,0.5918)(0.741,0.5972)(0.743,0.5989)(0.744,0.6583)(0.746,0.6558)(0.748,0.6873)(0.749,0.7442)(0.751,0.7432)(0.753,0.7745)(0.755,0.7653)(0.756,0.8044)(0.759,0.7869)(0.761,0.7788)(0.762,0.8015)(0.764,0.8529)(0.766,0.8792)(0.768,0.8993)(0.769,0.9297)(0.772,0.8658)(0.774,0.8538)(0.776,0.8809)(0.778,0.8767)(0.780,0.8724)(0.781,0.9232)(0.783,0.9844)(0.785,0.9731)(0.787,0.9904)(0.788,1.0204)(0.791,0.9497)(0.793,0.9624)(0.795,0.9183)(0.797,0.9385)(0.799,0.9078)(0.802,0.8943)(0.804,0.8516)(0.806,0.8424)(0.808,0.8316)(0.810,0.8758)(0.812,0.8680)(0.814,0.8690)(0.815,0.9005)(0.817,0.9301)(0.819,0.9398)(0.821,0.9626)(0.824,0.8710)(0.826,0.8683)(0.828,0.8108)(0.831,0.7617)(0.833,0.7828)(0.835,0.7293)(0.837,0.7292)(0.839,0.7227)(0.841,0.7391)(0.844,0.6653)(0.846,0.6761)(0.847,0.7042)(0.850,0.6848)(0.851,0.7116)(0.853,0.7150)(0.856,0.6512)(0.858,0.6693)(0.860,0.6960)(0.862,0.6948)(0.864,0.6668)(0.865,0.7180)(0.867,0.7326)(0.869,0.7955)(0.871,0.7972)(0.872,0.8508)(0.875,0.7431)(0.877,0.7655)(0.878,0.8060)(0.881,0.7656)(0.882,0.8096)(0.885,0.7536)(0.887,0.7652)(0.889,0.7404)(0.890,0.8040)(0.893,0.7695)(0.895,0.7759)(0.897,0.7739)(0.899,0.7428)(0.902,0.6318)(0.904,0.6610)(0.906,0.6332)(0.909,0.5423)(0.911,0.5239)(0.913,0.5235)(0.915,0.5598)(0.917,0.5451)(0.919,0.5516)(0.921,0.5522)(0.922,0.6075)(0.925,0.5871)(0.926,0.6655)(0.928,0.6408)(0.931,0.5913)(0.934,0.4703)(0.936,0.4705)(0.938,0.4558)(0.939,0.5138)(0.942,0.4245)(0.944,0.4261)(0.946,0.4632)(0.948,0.4093)(0.951,0.3976)(0.953,0.3839)(0.956,0.2818)(0.958,0.2224)(0.960,0.2123)(0.963,0.1575)(0.966,0.0312)(0.968,0.0961)(0.969,0.1523)(0.971,0.1937)(0.973,0.1871)(0.975,0.1160)(0.977,0.1083)(0.979,0.1086)(0.982,0.0750)(0.984,0.0688)(0.986,0.0820)(0.988,0.0491)(0.990,0.0768)(0.992,0.0530)(0.995,-0.018)(0.997,-0.024)(0.999,-0.012)(1,0)
		};
		
		
		\end{axis}
			\end{tikzpicture}
		\end{subfigure}}
		\centering{}\resizebox{.32\linewidth}{!}{
		\begin{subfigure}[b]{.75\linewidth}
			\begin{tikzpicture} 
			\begin{axis} 
		[
		ymin=-1.7,
		ymax=1.7,
		xmin=0,
		xmax=1,
		xlabel={\LARGE $F(t)$},
		ylabel={\LARGE $\sqrt{n}(F_n(t)-F(t))$},
		width=12cm,
		height=7cm,
		axis on top=true,
		axis x line=middle,
		axis y line=middle,	
		legend style = {at={(1,.1)}, anchor=south east}
		]
		
		\addplot[
		dashed,		
		]
		coordinates{
		(0,1.2221)(1,1.2221)
		};
		\addplot[
		dashed,		
		]
		coordinates{
		(0,-1.2221)(1,-1.2221)
		};
		
		\addplot[
		solid,
		color=black,
		]
		coordinates {
		(0,0)(0.001,-0.002)(0.002,0.0129)(0.005,-0.053)(0.007,-0.013)(0.009,-0.048)(0.011,-0.023)(0.012,0.0736)(0.014,0.0082)(0.017,-0.001)(0.018,0.0434)(0.021,-0.018)(0.024,-0.103)(0.026,-0.112)(0.028,-0.134)(0.029,-0.083)(0.032,-0.126)(0.034,-0.179)(0.037,-0.230)(0.038,-0.147)(0.040,-0.196)(0.044,-0.327)(0.046,-0.345)(0.048,-0.320)(0.050,-0.322)(0.052,-0.395)(0.055,-0.418)(0.056,-0.374)(0.059,-0.485)(0.061,-0.465)(0.063,-0.445)(0.065,-0.470)(0.068,-0.524)(0.070,-0.502)(0.071,-0.465)(0.073,-0.473)(0.075,-0.480)(0.078,-0.544)(0.080,-0.549)(0.082,-0.544)(0.084,-0.502)(0.086,-0.539)(0.088,-0.537)(0.090,-0.507)(0.092,-0.527)(0.094,-0.563)(0.096,-0.579)(0.099,-0.645)(0.101,-0.698)(0.103,-0.659)(0.105,-0.678)(0.107,-0.672)(0.109,-0.623)(0.112,-0.700)(0.113,-0.633)(0.115,-0.600)(0.116,-0.557)(0.117,-0.495)(0.119,-0.478)(0.121,-0.454)(0.123,-0.419)(0.125,-0.475)(0.127,-0.465)(0.129,-0.466)(0.132,-0.509)(0.133,-0.499)(0.135,-0.430)(0.136,-0.363)(0.138,-0.364)(0.141,-0.437)(0.143,-0.417)(0.145,-0.404)(0.146,-0.396)(0.149,-0.429)(0.150,-0.375)(0.153,-0.443)(0.154,-0.377)(0.156,-0.331)(0.158,-0.321)(0.160,-0.314)(0.162,-0.367)(0.164,-0.313)(0.167,-0.458)(0.168,-0.394)(0.171,-0.426)(0.173,-0.407)(0.174,-0.380)(0.178,-0.508)(0.179,-0.479)(0.181,-0.413)(0.183,-0.403)(0.185,-0.434)(0.188,-0.511)(0.190,-0.506)(0.191,-0.435)(0.193,-0.445)(0.195,-0.411)(0.197,-0.429)(0.199,-0.425)(0.200,-0.395)(0.203,-0.467)(0.205,-0.427)(0.206,-0.387)(0.208,-0.324)(0.210,-0.302)(0.211,-0.255)(0.212,-0.194)(0.214,-0.185)(0.217,-0.242)(0.219,-0.218)(0.221,-0.222)(0.223,-0.214)(0.224,-0.165)(0.226,-0.155)(0.229,-0.260)(0.232,-0.311)(0.234,-0.307)(0.235,-0.284)(0.237,-0.273)(0.239,-0.269)(0.241,-0.296)(0.244,-0.325)(0.245,-0.298)(0.248,-0.392)(0.251,-0.416)(0.253,-0.451)(0.254,-0.398)(0.256,-0.380)(0.258,-0.377)(0.260,-0.378)(0.262,-0.373)(0.264,-0.335)(0.266,-0.359)(0.268,-0.308)(0.270,-0.316)(0.272,-0.375)(0.275,-0.434)(0.277,-0.427)(0.279,-0.431)(0.280,-0.372)(0.282,-0.387)(0.285,-0.475)(0.288,-0.548)(0.291,-0.613)(0.293,-0.626)(0.295,-0.603)(0.296,-0.518)(0.298,-0.534)(0.299,-0.489)(0.301,-0.484)(0.303,-0.455)(0.305,-0.405)(0.306,-0.361)(0.308,-0.363)(0.310,-0.317)(0.312,-0.328)(0.313,-0.281)(0.316,-0.304)(0.317,-0.248)(0.319,-0.206)(0.321,-0.286)(0.323,-0.203)(0.324,-0.170)(0.326,-0.185)(0.328,-0.189)(0.330,-0.175)(0.332,-0.174)(0.334,-0.198)(0.337,-0.229)(0.339,-0.239)(0.341,-0.273)(0.343,-0.268)(0.345,-0.274)(0.347,-0.213)(0.350,-0.356)(0.352,-0.339)(0.354,-0.315)(0.356,-0.311)(0.358,-0.322)(0.360,-0.359)(0.361,-0.293)(0.363,-0.299)(0.365,-0.257)(0.367,-0.243)(0.368,-0.176)(0.370,-0.170)(0.372,-0.106)(0.373,-0.058)(0.376,-0.132)(0.377,-0.055)(0.379,-0.011)(0.381,-0.027)(0.383,-0.073)(0.385,-0.019)(0.386,0.0470)(0.388,0.0397)(0.391,-0.003)(0.393,-0.060)(0.395,-0.040)(0.397,-0.044)(0.399,-0.092)(0.401,-0.053)(0.404,-0.159)(0.406,-0.141)(0.407,-0.071)(0.410,-0.123)(0.411,-0.094)(0.414,-0.140)(0.416,-0.156)(0.418,-0.135)(0.420,-0.163)(0.422,-0.108)(0.423,-0.063)(0.425,-0.056)(0.427,-0.019)(0.428,0.0134)(0.430,0.0716)(0.432,0.0998)(0.433,0.1042)(0.436,0.0728)(0.437,0.1486)(0.439,0.1694)(0.440,0.2426)(0.442,0.2849)(0.444,0.2904)(0.446,0.2746)(0.447,0.3117)(0.449,0.3594)(0.451,0.3554)(0.453,0.3930)(0.454,0.4529)(0.456,0.4823)(0.458,0.4237)(0.460,0.4491)(0.463,0.3518)(0.465,0.3465)(0.466,0.4020)(0.470,0.2743)(0.472,0.2876)(0.473,0.3276)(0.475,0.3828)(0.477,0.3662)(0.479,0.3208)(0.482,0.2909)(0.484,0.2646)(0.487,0.1999)(0.488,0.2634)(0.490,0.2760)(0.492,0.2836)(0.494,0.2759)(0.496,0.2921)(0.498,0.2920)(0.500,0.2572)(0.503,0.1902)(0.504,0.2527)(0.506,0.2736)(0.508,0.2730)(0.509,0.3335)(0.511,0.3517)(0.513,0.3143)(0.516,0.2617)(0.518,0.2849)(0.520,0.2345)(0.522,0.2017)(0.525,0.1775)(0.527,0.1601)(0.529,0.1983)(0.531,0.1674)(0.533,0.1231)(0.536,0.0870)(0.538,0.0653)(0.540,0.0585)(0.542,0.0545)(0.544,0.0306)(0.546,0.0713)(0.549,-0.031)(0.550,0.0055)(0.552,0.0361)(0.554,0.0043)(0.556,0.0116)(0.559,-0.080)(0.562,-0.135)(0.564,-0.141)(0.565,-0.067)(0.567,-0.064)(0.570,-0.110)(0.571,-0.047)(0.573,-0.066)(0.575,-0.082)(0.577,-0.075)(0.579,-0.036)(0.581,-0.089)(0.583,-0.081)(0.586,-0.147)(0.589,-0.205)(0.591,-0.201)(0.593,-0.296)(0.596,-0.323)(0.598,-0.399)(0.600,-0.370)(0.603,-0.420)(0.605,-0.435)(0.607,-0.432)(0.609,-0.472)(0.611,-0.476)(0.613,-0.435)(0.615,-0.468)(0.618,-0.554)(0.620,-0.526)(0.622,-0.568)(0.624,-0.529)(0.626,-0.564)(0.628,-0.525)(0.629,-0.490)(0.631,-0.499)(0.633,-0.460)(0.635,-0.406)(0.636,-0.398)(0.638,-0.381)(0.640,-0.356)(0.642,-0.340)(0.644,-0.338)(0.646,-0.374)(0.649,-0.448)(0.652,-0.574)(0.654,-0.553)(0.656,-0.563)(0.658,-0.555)(0.660,-0.534)(0.661,-0.470)(0.663,-0.490)(0.666,-0.583)(0.668,-0.569)(0.670,-0.586)(0.672,-0.564)(0.673,-0.493)(0.674,-0.392)(0.677,-0.405)(0.679,-0.483)(0.682,-0.514)(0.683,-0.472)(0.685,-0.492)(0.687,-0.417)(0.689,-0.420)(0.691,-0.430)(0.693,-0.477)(0.696,-0.532)(0.698,-0.576)(0.700,-0.584)(0.702,-0.581)(0.704,-0.552)(0.705,-0.493)(0.707,-0.441)(0.710,-0.504)(0.711,-0.473)(0.714,-0.516)(0.716,-0.510)(0.718,-0.572)(0.720,-0.576)(0.722,-0.525)(0.723,-0.474)(0.725,-0.468)(0.727,-0.478)(0.729,-0.472)(0.731,-0.404)(0.732,-0.384)(0.734,-0.355)(0.736,-0.322)(0.738,-0.367)(0.740,-0.372)(0.742,-0.343)(0.745,-0.430)(0.747,-0.422)(0.748,-0.390)(0.750,-0.376)(0.754,-0.500)(0.755,-0.497)(0.757,-0.417)(0.759,-0.481)(0.761,-0.495)(0.763,-0.426)(0.765,-0.403)(0.767,-0.426)(0.769,-0.478)(0.771,-0.443)(0.773,-0.440)(0.775,-0.417)(0.777,-0.402)(0.778,-0.389)(0.781,-0.471)(0.783,-0.469)(0.785,-0.436)(0.787,-0.446)(0.789,-0.467)(0.790,-0.384)(0.792,-0.340)(0.793,-0.293)(0.796,-0.335)(0.798,-0.385)(0.800,-0.385)(0.802,-0.351)(0.804,-0.362)(0.808,-0.518)(0.810,-0.553)(0.812,-0.505)(0.814,-0.514)(0.816,-0.511)(0.817,-0.494)(0.820,-0.541)(0.822,-0.576)(0.824,-0.522)(0.826,-0.513)(0.827,-0.481)(0.829,-0.433)(0.831,-0.402)(0.833,-0.417)(0.835,-0.407)(0.838,-0.534)(0.840,-0.540)(0.842,-0.581)(0.845,-0.603)(0.846,-0.587)(0.849,-0.601)(0.850,-0.588)(0.852,-0.583)(0.855,-0.600)(0.856,-0.545)(0.858,-0.557)(0.859,-0.474)(0.861,-0.477)(0.863,-0.437)(0.865,-0.474)(0.867,-0.434)(0.870,-0.517)(0.872,-0.508)(0.873,-0.493)(0.875,-0.455)(0.877,-0.471)(0.880,-0.513)(0.881,-0.493)(0.883,-0.476)(0.885,-0.477)(0.887,-0.404)(0.888,-0.388)(0.891,-0.422)(0.892,-0.345)(0.893,-0.257)(0.895,-0.254)(0.897,-0.223)(0.898,-0.147)(0.900,-0.143)(0.902,-0.129)(0.904,-0.135)(0.906,-0.157)(0.908,-0.134)(0.909,-0.083)(0.911,-0.033)(0.913,-0.087)(0.915,-0.058)(0.917,-0.050)(0.919,-0.017)(0.921,-0.072)(0.924,-0.135)(0.926,-0.103)(0.927,-0.061)(0.930,-0.124)(0.931,-0.090)(0.934,-0.185)(0.936,-0.114)(0.938,-0.150)(0.940,-0.167)(0.942,-0.172)(0.944,-0.116)(0.946,-0.145)(0.948,-0.127)(0.950,-0.113)(0.951,-0.072)(0.953,-0.044)(0.956,-0.159)(0.957,-0.089)(0.959,-0.028)(0.961,-0.015)(0.963,-0.031)(0.964,0.0069)(0.967,-0.042)(0.969,-0.030)(0.970,0.0318)(0.973,-0.041)(0.975,-0.060)(0.977,-0.076)(0.979,-0.073)(0.981,-0.093)(0.983,-0.096)(0.985,-0.056)(0.987,-0.080)(0.990,-0.119)(0.991,-0.042)(0.992,0.0152)(0.994,0.0193)(0.996,0.0241)(0.999,-0.025)(1,0)
		};
		
		
		\end{axis}
			\end{tikzpicture}
		\end{subfigure}}
\caption{\footnotesize
Empirical distribution functions for spectrally negative infinite variation 
stable process with parameters $\rho=1/\alpha$ and, from left to right, 
$\alpha=1.1$, $\alpha=1.5$ and $\alpha=1.9$. The top graphs show the 
empirical distribution functions $F_N$ for $N=10^4$ samples and compare it 
to the distribution function 
$F=\overline{S}(\alpha,\rho)=S^+(\alpha,\rho)$~\citep[Thm~1]{MR3033593}. 
The bottom graphs show $s\mapsto\sqrt{N}(F_N\circ F^{-1}(s)-s)$ on $[0,1]$
(or equivalently, the curve $t\mapsto(F(t),\sqrt{N}(F_N(t)-F(t)))$ in 
$\mathbb{R}^2$ for $t>0$), which converges weakly to a Brownian bridge. 
The dashed lines are the $0.05$ and $0.95$ quantiles of the Kolmogorov-Smirnov 
statistic, derived from the distribution of the signed maximum modulus of the Brownian bridge. 
}
\label{fig:ecdf}
\end{figure}
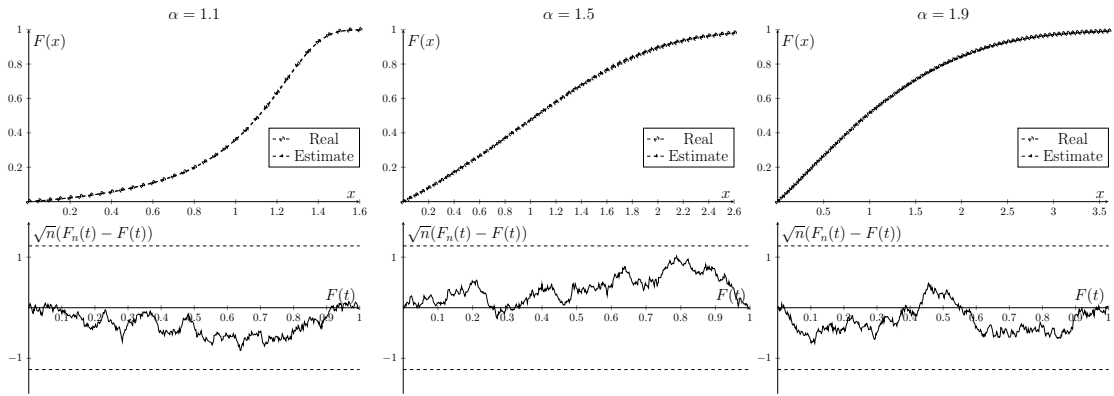

\begin{algorithm}
	\caption{Perfect simulation of $X_0\overset{d}{=}\overline{Y}_1$}
	\label{alg:FullAlgorithm}
	\begin{algorithmic}[1]
		\Require{Parameters $0<\delta<d<\frac{1}{\alpha\rho}$,
			$\kappa>\max\{\frac{\log(2)}{3\eta},\frac{1}{\alpha\rho}\}$, 
			$\gamma>0$ and $\Delta(0)\in\mathcal{Z}^0$}
		\State{Put $x:=\infty$, $t:=1$, $s:=\Delta(0)$ and $m:=n:=\Delta(0)+1$}
		\Comment{$x$ is an upper bound on $\{C_k\}_{k\in\mathcal{Z}^{\Delta(0)}}$}
		\State{Sample $\{\Theta_k\}_{k\in\mathcal{Z}^0_{\Delta(0)}}$}
		\Comment{Recall its definition in Section~\ref{sec:Backward_Sim}, paragraph 1}
		\Loop
		\State{Sample $\big(\chi_{m-1},
			\{S_k\}_{k\in\mathcal{Z}_{\chi_{m-1}}^{s}}\big)$}
		\Comment{Algorithm~\ref{alg:Algorithm 3}}
		\While{$n = m$ or $\Delta(t) > \chi_{m-1}$}
		\State{Sample $C_{\Delta(t)}, \ldots, C_{T^{\Delta(t)}_{-2\kappa}}$ conditional on 
			$\{R_{\Delta(t)}<x-C_{\Delta(t)}\}$}
		\Comment{Algorithm~\ref{alg:SubAlgorithm 5-3}}
		\State{Put $\Delta(t+1):=T^{\Delta(t)}_{-2\kappa}$ and $t:=t+1$}
		\Comment{Recall its definition in~\eqref{eq:T_x^k}}
		\State{Sample $1_{\{R_{\Delta(t)}>\kappa\}}$ given $\{R_{\Delta(t)} < x-C_{\Delta(t)}\}$}
		\Comment{Algorithm~\ref{alg:SubAlgorithm 5-1}}
		\If{$1_{\{R_{\Delta(t)}>\kappa\}}=0$}
		\State{Compute $\{ R_{k}\}_{k\in\mathcal{Z}^n_{\Delta(t-1)}}$ and 
			put $n:=\Delta(t-1)$ and $x:=C_{\Delta(t)}+\kappa$}
		\Else
		\State{Sample $C_{\Delta(t)-1}, \ldots, C_{T_\kappa^{\Delta(t)}}$ from $\mathbb{P}^\eta$}
		\Comment{Algorithm~\ref{alg:SubAlgorithm 5-2}}
		\State{Put $\Delta(t+1):=T_\kappa^{\Delta(t)}$ and $t:= t+1$}
		\EndIf
		\EndWhile
		\State{Sample $\{(U_k,\Lambda_k)\}_{k\in\mathcal{Z}_{\chi_{m-1}}^{s}}$ from $\{F_{k}\}_{k\in\mathcal{Z}_{\chi_{m-1}}^{s}}$ 
			and put $s:=\chi_{m-1}$}
		\Comment{Algorithm~\ref{alg:Algorithm 2-1}}
		\State{Compute $D_{m-1}$ and put $m:=m-1$}
		\Comment{Recall its definition in~\eqref{eq:DominatingProcess}}
		\If{$D_{m}\leq a(\Theta_{m})$}
		\State{\Return $X_0:=\psi(\cdots\psi(D_{m},\Theta_{m}),\cdots,\Theta_{-1})$}
		\Comment{In this case $\sigma=m$}
		\ElsIf{$m=\Delta(0)$}
		\State{Put $X_0:=\psi(\cdots\psi(D_m,\Theta_m),\cdots,\Theta_{-1})$}
		\If{coalescence was detected}
		\State{\Return $X_0$}
		\EndIf
		\EndIf
		\EndLoop
	\end{algorithmic}
\end{algorithm}

\subsection{\label{subsec:parameter_fit}Parameter choice and numerical performance}
As explicitly stated in Algorithm~\ref{alg:FullAlgorithm}, and if 
one allows $m^\ast$ (recall its definition in paragraph~1, p.~14) 
to vary over $\lfloor\frac{1}{\delta\gamma}\log\mathbb{E} S_1^\gamma\rfloor^++\mathbb{N}$,
our simulation procedure has 6 different parameters. A full 
theoretical optimisation is infeasible as it heavily depends on, 
among other things, the way the algorithm is coded, the computational 
cost of simulating each variable, the cost of each calculation, 
memory accessing cost, the quality and state of the RAM and $(\alpha,\rho)$. 
However, for the sake of presenting its practical feasibility, 
we have implemented the algorithm in the Julia programming language 
(see~\cite{Jorge_GitHub}) and ran it on a macOS Mojave 10.14.3 (18D109)
with a 4.2 GHz Intel\textregistered Core\texttrademark i7 processor and a 
8 GB 2400 MHz DDR4 memory. This implementation is far from optimal, but still 
outputs $10^4$ samples in approximately $1.15$ seconds (without multithreading) 
for the suggested parameters $(d,\delta,\gamma,\kappa,\Delta(0),m^\ast)=\varpi$ where 
\[
\varpi=\varpi(\alpha,\rho)
=\left(\frac{2}{3\alpha\rho},\frac{1}{3\alpha\rho},\frac{19}{20}\alpha,
4+\max\left\{\frac{\log(2)}{3\eta(\frac{2}{3\alpha\rho})},\frac{1}{\alpha\rho}\right\},40,
12+\left\lfloor\frac{60}{19}\rho\log\mathbb{E} S_1^{\frac{19}{20}\alpha}\right\rfloor^+\right).
\] 
This performance varies slightly for different choices of $(\alpha,\rho)$.
To put things in perspective, Algorithm~\ref{alg:Algorithm 2-1} outputs,
for the parameter choice $\varpi$, $10^6$ samples in approximately $0.4322$ 
seconds and drawing $10^6$ samples from $S^+(\alpha,\rho)$ takes $0.1833$ seconds.
On the other hand, the first iteration of Algorithm~\ref{alg:Algorithm 3} 
(which simulates the indicators $\{I_k^0\}_{k\in\mathcal{Z}^0_{\chi_0}}$ 
and the conditionally positive stable random variables 
$\{S_k\}_{k\in\mathcal{Z}^0_{\chi_0}}$) simulates $10^4$ samples in about 
$0.8125$ seconds and is, although fast, the most computationally costly 
component of Algorithm~\ref{alg:FullAlgorithm}. 
The main sources of this cost are the calculation of the probabilities 
$\{p(m)\}$ (see their definition in~\eqref{eq:probTails}) and the simulation of 
$\{S_k\}_{k\in\mathcal{Z}^0_{\chi_0}}$ conditioned on the values 
of $\{I_k^0\}_{k\in\mathcal{Z}^0_{\chi_0}}$.

Next we show the local marginal behaviour of the number of samples outputted 
with confidence intervals, for a few different choices of parameters 
$(\alpha,\rho)$. We will see that, although $\varpi$ may not be optimal, 
it is a simple and yet efficient choice. Moreover, the variation in 
performance for parameters close to this one is small, thus showing 
that this choice is relatively robust.

It should be noted that the data presented in Figure~\ref{fig:marginals} 
is dependent on the characteristics of the hardware and software used.
Hence, these exact numbers are not easily replicated. For instance, these 
times scale sub-linearly as a function of the batch size, and are not 
replicated despite using the garbage collector and the same random seed. 
It is readily seen that the exact value of the parameters $d$, $\delta$, 
$\Delta(0)$ and $m^\ast$ is not too important in so far as they remain at a 
reasonable distance from their boundaries (where $\infty$ is a right-boundary 
for $\Delta(0)$ and $m^\ast$). 
The value of $\kappa$ is slightly more sensitive, as is $\gamma$. 
Other choices of $(\alpha,\rho)$ have slightly different behaviours. 
The shapes of these curves are similar, but the apparent minima change. 
Thus, we argue that $\varpi$ is a simple yet sensible choice.

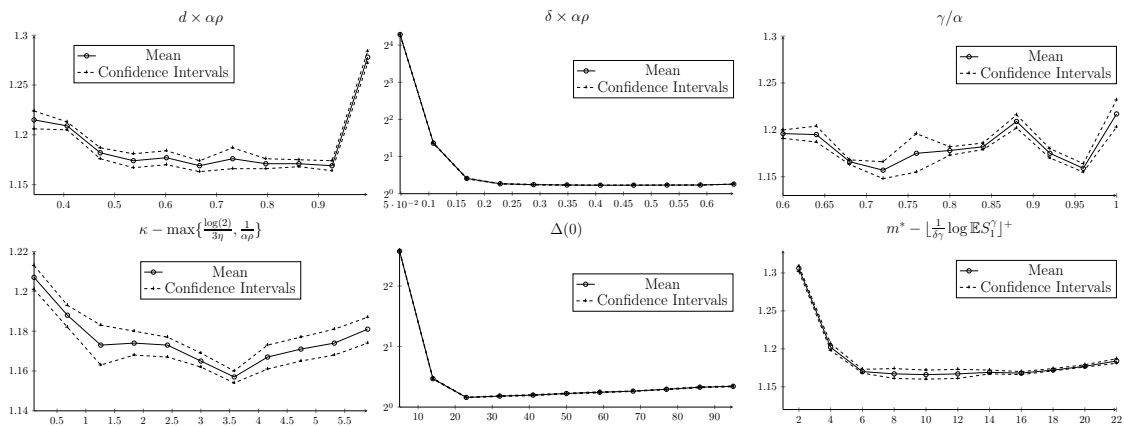
\begin{figure}[H]
	\centering{}\resizebox{.32\linewidth}{!}{
		\begin{subfigure}[h]{.78\linewidth}
			\begin{tikzpicture} 
			\begin{axis} 
			[
			ymax = 1.3,
			ymin = 1.14,
			title = {\LARGE$d\times \alpha\rho$},
			width=12.5cm,
			height=6.8cm,
			axis on top=true,
			axis x line=bottom,
			axis y line=left,
			legend style = {at={(.6,.7)}, anchor=south east}
			]
			
			\addplot[
			solid,
			mark=o,
			color=black,
			]
			coordinates {
				(0.3417,1.215)(0.4067,1.209)(0.4717,1.182)(0.5367,1.174)(0.6017,1.177)(0.6667,1.169)(0.7317,1.176)(0.7967,1.171)(0.8617,1.171)(0.9267,1.169)(0.9967,1.278)
			};
			
			\addplot[
			dashed,
			mark=+,
			color=black,
			]
			coordinates {
				(0.3417,1.224)(0.4067,1.213)(0.4717,1.187)(0.5367,1.181)(0.6017,1.184)(0.6667,1.174)(0.7317,1.187)(0.7967,1.176)(0.8617,1.175)(0.9267,1.174)(0.9967,1.284)
			};
			
			\addplot[
			dashed,
			mark=+,
			color=black,
			]
			coordinates {
				(0.3417,1.206)(0.4067,1.205)(0.4717,1.176)(0.5367,1.167)(0.6017,1.170)(0.6667,1.163)(0.7317,1.166)(0.7967,1.166)(0.8617,1.168)(0.9267,1.164)(0.9967,1.272)
			};
			\legend {\LARGE Mean,\LARGE Confidence Intervals};
			\end{axis}
			\end{tikzpicture}
	\end{subfigure}}
	\centering{}\resizebox{.32\linewidth}{!}{
		\begin{subfigure}[h]{.78\linewidth}
			\begin{tikzpicture} 
			\begin{semilogyaxis} 
			[
			log basis y=2,
			ymin = 1,
			title = {\LARGE$\delta\times \alpha\rho$},
			width=12.5cm,
			height=6.8cm,
			axis on top=true,
			axis x line=bottom,
			axis y line=left,
			legend style = {at={(1,.6)}, anchor=south east}
			]
			
			\addplot[
			solid,
			mark=o,
			color=black,
			]
			coordinates {
				(0.04833,19.49)(0.1083,2.565)(0.1683,1.328)(0.2283,1.203)(0.2883,1.184)(0.3483,1.175)(0.4083,1.171)(0.4683,1.172)(0.5283,1.174)(0.5883,1.176)(0.6483,1.195)		
			};
			
			\addplot[
			dashed,
			mark=+,
			color=black,
			]
			coordinates {
				(0.04833,19.55)(0.1083,2.583)(0.1683,1.338)(0.2283,1.207)(0.2883,1.191)(0.3483,1.180)(0.4083,1.172)(0.4683,1.173)(0.5283,1.175)(0.5883,1.178)(0.6483,1.197)
			};
			
			\addplot[
			dashed,
			mark=+,
			color=black,
			]
			coordinates {
				(0.04833,19.44)(0.1083,2.548)(0.1683,1.318)(0.2283,1.199)(0.2883,1.176)(0.3483,1.170)(0.4083,1.170)(0.4683,1.171)(0.5283,1.172)(0.5883,1.174)(0.6483,1.193)
			};
			\legend {\LARGE Mean,\LARGE Confidence Intervals};
			\end{semilogyaxis}
			\end{tikzpicture}
	\end{subfigure}}
	\centering{}\resizebox{.32\linewidth}{!}{
		\begin{subfigure}[h]{.78\linewidth}
			\begin{tikzpicture} 
			\begin{axis} 
			[
			ymin=1.13,
			ymax=1.3,
			xmin=0.6,
			xmax=1,
			title={\LARGE$\gamma/\alpha$},
			width=12.5cm,
			height=6.8cm,
			axis on top=true,
			axis x line=bottom,
			axis y line=left,
			legend style = {at={(1,.7)}, anchor=south east}
			]
			
			\addplot[
			solid,
			mark=o,
			color=black,
			]
			coordinates {
				(0.6,1.196)(0.64,1.195)(0.68,1.166)(0.72,1.157)(0.76,1.175)(0.8,1.178)(0.84,1.182)(0.88,1.209)(0.92,1.175)(0.96,1.159)(.9999,1.217)
			};
			
			\addplot[
			dashed,
			mark=+,
			color=black,
			]
			coordinates {
				(0.6,1.2)(0.64,1.204)(0.68,1.168)(0.72,1.166)(0.76,1.196)(0.8,1.182)(0.84,1.186)(0.88,1.216)(0.92,1.18)(0.96,1.164)(.9999,1.232)
			};
			
			\addplot[
			dashed,
			mark=+,
			color=black,
			]
			coordinates {
				(0.6,1.191)(0.64,1.187)(0.68,1.163)(0.72,1.148)(0.76,1.155)(0.8,1.173)(0.84,1.179)(0.88,1.202)(0.92,1.17)(0.96,1.155)(.9999,1.203)
			};
			\legend {\LARGE Mean,\LARGE Confidence Intervals};
			\end{axis}
			\end{tikzpicture}
	\end{subfigure}}
	\centering{}\resizebox{.32\linewidth}{!}{
		\begin{subfigure}[h]{.78\linewidth}
			\begin{tikzpicture} 
			\begin{axis} 
			[
			ymin=1.14,
			ymax=1.22,
			xmin=.1,
			xmax=5.9,
			title={\LARGE$\kappa-\max\{\frac{\log(2)}{3\eta},\frac{1}{\alpha\rho}\}$},
			width=12.5cm,
			height=6.8cm,
			axis on top=true,
			axis x line=bottom,
			axis y line=left,
			legend style = {at={(.8,.7)}, anchor=south east}
			]
			
			\addplot[
			solid,
			mark=o,
			color=black,
			]
			coordinates {
				(0.1,1.207)(0.68,1.188)(1.26,1.173)(1.84,1.174)(2.42,1.173)(3.0,1.165)(3.58,1.157)(4.16,1.167)(4.74,1.171)(5.32,1.174)(5.9,1.181)
			};
			
			\addplot[
			dashed,
			mark=+,
			color=black,
			]
			coordinates {
				(0.1,1.213)(0.68,1.193)(1.26,1.183)(1.84,1.18)(2.42,1.177)(3.0,1.169)(3.58,1.160)(4.16,1.173)(4.74,1.177)(5.32,1.181)(5.9,1.187)
			};
			
			\addplot[
			dashed,
			mark=+,
			color=black,
			]
			coordinates {
				(0.1,1.201)(0.68,1.182)(1.26,1.163)(1.84,1.168)(2.42,1.167)(3.0,1.162)(3.58,1.154)(4.16,1.161)(4.74,1.165)(5.32,1.168)(5.9,1.174)
			};
			\legend {\LARGE Mean,\LARGE Confidence Intervals};
			\end{axis}
			\end{tikzpicture}
	\end{subfigure}}
	\centering{}\resizebox{.32\linewidth}{!}{
		\begin{subfigure}[h]{.78\linewidth}
			\begin{tikzpicture} 
			\begin{semilogyaxis} 
			[
			log basis y=2,
			ymin=1,
			ymax=6.2,
			xmin=5,
			xmax=95,
			title={\LARGE$\Delta(0)$},
			width=12.5cm,
			height=6.8cm,
			axis on top=true,
			axis x line=bottom,
			axis y line=left,
			legend style = {at={(1,.6)}, anchor=south east}
			]
			
			\addplot[
			solid,
			mark=o,
			color=black,
			]
			coordinates {
				(5,5.958)(14,1.384)(23,1.117)(32,1.134)(41,1.147)(50,1.168)(59,1.185)(68,1.2)(77,1.226)(86,1.255)(95,1.268)
			};
			
			\addplot[
			dashed,
			mark=+,
			color=black,
			]
			coordinates {
				(5,5.99)(14,1.393)(23,1.12)(32,1.14)(41,1.153)(50,1.173)(59,1.192)(68,1.202)(77,1.231)(86,1.262)(95,1.274)
			};
			
			\addplot[
			dashed,
			mark=+,
			color=black,
			]
			coordinates {
				(5,5.925)(14,1.376)(23,1.113)(32,1.128)(41,1.14)(50,1.162)(59,1.177)(68,1.197)(77,1.22)(86,1.248)(95,1.262)
			};
			\legend {\LARGE Mean,\LARGE Confidence Intervals};
			\end{semilogyaxis}
			\end{tikzpicture}
	\end{subfigure}}
	\centering{}\resizebox{.32\linewidth}{!}{
		\begin{subfigure}[h]{.78\linewidth}
			\begin{tikzpicture} 
			\begin{axis} 
			[
			log basis y=2,
			ymin = 1.12,
			ymax = 1.33,
			xmin=1,
			xmax=22,
			title={\LARGE$m^\ast-\lfloor\frac{1}{\delta\gamma}\log\mathbb{E} S_1^\gamma\rfloor^+$},
			width=12.5cm,
			height=6.8cm,
			axis on top=true,
			axis x line=bottom,
			axis y line=left,
			legend style = {at={(1,.7)}, anchor=south east}
			]
			
			\addplot[
			solid,
			mark=o,
			color=black,
			]
			coordinates {
				(2,1.306)(4,1.202)(6,1.17)(8,1.167)(10,1.166)(12,1.167)(14,1.169)(16,1.168)(18,1.172)(20,1.177)(22,1.184)
			};
			
			\addplot[
			dashed,
			mark=+,
			color=black,
			]
			coordinates {
				(2,1.309)(4,1.207)(6,1.173)(8,1.174)(10,1.172)(12,1.173)(14,1.172)(16,1.17)(18,1.174)(20,1.179)(22,1.187)
			};
			
			\addplot[
			dashed,
			mark=+,
			color=black,
			]
			coordinates {
				(2,1.302)(4,1.198)(6,1.168)(8,1.161)(10,1.16)(12,1.161)(14,1.167)(16,1.166)(18,1.171)(20,1.176)(22,1.181)
			};
			\legend {\LARGE Mean,\LARGE Confidence Intervals};
			\end{axis}
		\end{tikzpicture}
	\end{subfigure}}
	\caption{\footnotesize
		Time taken (in seconds) to simulate $N=10^4$ samples of $\overline{S}(1.3,1/2)$ with parameters moving about $\varpi$. In each plot, one parameter moves and all others are kept constant at the respective value of $\varpi$. We took $100$ batches of samples, each with $N=10^4$ independent simulations, to construct asymptotic $95\%$ confidence intervals based on the central limit theorem.
	}
\label{fig:marginals}
\end{figure}

\section*{Acknowledgements}
JGC and AM are supported by The Alan Turing Institute under the EPSRC grant EP/N510129/1;
AM supported by EPSRC grant EP/P003818/1 and the Turing Fellowship funded by the Programme on Data-Centric Engineering of Lloyd's Register Foundation;
GUB supported by CoNaCyT grant FC-2016-1946 and UNAM-DGAPA-PAPIIT grant IN115217; 
JGC supported by CoNaCyT scholarship 2018-000009-01EXTF-00624. 
We thank Stephen Connor for the reference~\cite{MR3683256}.

\appendix

\section{\label{sec:StableSampling}Sampling the marginals of stable processes}

A L\'{e}vy process $Y=(Y_{t})_{t\in[0,\infty)}$ in $\mathbb{R}$ is \emph{strictly
stable} with index $\alpha\in(0,2]$ if for any constant $c\geq0$ the
processes $\left(Y_{ct}\right)_{t\in[0,\infty)}$ and $\left(c^{1/\alpha}Y_{t}\right)_{t\in[0,\infty)}$
have the same law. For brevity, we call $Y$ a \emph{stable process}. 
Sampling the increments of $Y$ hence reduces to sampling
$Y_{1}$. Using Zolotarev's (C) form~\citep{MR1745764}, up to a scaling constant the law of
$Y_{1}$ is parametrised by $\left(\alpha,\beta\right)\in\left(0,2\right]\times\left[-1,1\right]$
via 
\begin{eqnarray}
\mathbb{E} e^{itY_{1}} & = & \exp\left(-\left|t\right|^{\alpha}e^{-i\frac{\pi\alpha}{2}\theta\text{sgn}\left(t\right)}\right),\quad\text{where }t\in\mathbb{R},\ \theta=\beta\left(1_{\alpha\leq1}+\frac{\alpha-2}{\alpha}1_{\alpha>1}\right),\label{eq:C_form}
\end{eqnarray}
and $\text{sgn}(t)$ equals $1$ (resp. $-1$) if $t\geq0$ (resp.
$t<0$). The Mellin transform of $Y_{1}$ equals 
\begin{equation}
\mathbb{E} Y_{1}^{s}1_{Y_{1}>0}=\rho\frac{\Gamma\left(1+s\right)\Gamma\left(1-\frac{s}{\alpha}\right)}{\Gamma\left(1+s\rho\right)\Gamma\left(1-s\rho\right)},\label{eq:MellinStable}
\end{equation}
where $\rho=\frac{1+\theta}{2}$ and $\Gamma(\cdot)$ denotes the
gamma function (see~\citep{MR1745764} Section 5.6). Taking $s=0$
in~\eqref{eq:MellinStable} implies that the stable law is uniquely
determined by $\alpha$ and its positivity parameter $\rho=\mathbb{P}\left(Y_{1}>0\right)$.
If $\alpha>1$, 
the pair $\left(\alpha,\rho\right)\in(0,2]\times[0,1]$
must satisfy $\rho\in\left[1-\frac{1}{\alpha},\frac{1}{\alpha}\right]$,
since $\theta\in[1-\frac{2}{\alpha},\frac{2}{\alpha}-1]$.

Let $S\left(\alpha,\rho\right)$ and $S^{+}\left(\alpha,\rho\right)$
denote the laws of $Y_{1}$ and $Y_{1}$ conditioned on being positive,
respectively. As $\rho,\alpha\rho\in[0,1]$ and the Mellin transform
determines the law uniquely,~\eqref{eq:MellinStable} implies that
$\left(Z'/Z''\right)^{\rho}$ follows $S^{+}\left(\alpha,\rho\right)$,
where $Z'\sim S\left(\alpha\rho,1\right)$ and $Z''\sim S\left(\rho,1\right)$
are independent. Since $P'B+P''\left(1-B\right)$
follows $S\left(\alpha,\rho\right)$, 
where $P'\sim S^{+}\left(\alpha,\rho\right)$,
$P'\sim S^{+}\left(\alpha,1-\rho\right)$ and $B\sim Ber\left(\rho\right)$
are independent, 
we need only be able to simulate a positive stable random variable with law $S\left(\alpha,1\right)$
for any $\alpha\in(0,1]$. If $\alpha=1$, then by~\eqref{eq:C_form},
$Y_{1}$ is a constant equal to one. If $\alpha\in(0,1)$, Kanter's
factorisation states
\[
\left(\sin\left(\alpha\pi U\right)^{\alpha}\sin\left(\left(1-\alpha\right)\pi U\right)^{1-\alpha}/\sin\left(\pi U\right)\right)^{\frac{1}{\alpha}}E^{1-\frac{1}{\alpha}}\sim S\left(\alpha,1\right),
\]
where 
$E$ is exponential with mean one, independent of 
$U$, which  is uniform on $(0,1)$ 
(see~\citep[Sec~4.4]{MR1745764}). For
alternative ways of sampling from the laws $S\left(\alpha,\rho\right)$
and $S^{+}\left(\alpha,\rho\right)$ we refer to~\citep{MR3233961}.

\bibliographystyle{amsalpha}

\end{document}